\documentclass[11pt]{amsart}
\usepackage{categorytheory}
\fullpage

\usepackage{amsfonts}
\usepackage{amsmath}
\usepackage{amssymb}
\usepackage[mathscr]{euscript}
\usepackage{tikz}
\usetikzlibrary{arrows}
\usetikzlibrary{decorations.pathmorphing}
\usetikzlibrary{decorations.markings}
\usetikzlibrary{shapes}
\usetikzlibrary{backgrounds}
\usetikzlibrary{patterns,snakes}
\usepackage[all]{xy}

\renewcommand{\hat}[1]{\ensuremath{\widehat{#1}}}

\newcommand{\ASB}{\mathbf{Asm}}    
\newcommand{\bs}{\smallsetminus}
\newcommand{\Cat}{\mathbf{Cat}}

\newcommand{\SCob}[2]{\{\uppercase{#1}_{#2}\}_{#2\in\uppercase{#2}}}
\newcommand{\smashes}{\smash\cdots\smash}
\newcommand{\SSub}{\mathrm{Sub}}
\renewcommand{\tilde}[1]{\ensuremath{\widetilde{#1}}}

\newcommand{\Tw}{\mathrm{Tw}}

\newcommand{\V}{\mathscr{V}}

\DeclareMathOperator{\coker}{coker}
\renewcommand{\R}{\mathbb{R}}
\newcommand{\gG}{\mathfrak{G}}

\newcommand{\ab}{\mathrm{ab}}

\newcommand{\W}{\mathcal{W}}
\newcommand{\SC}{\mathrm{SC}}
\newcommand{\curD}{\mathscr{D}}
\newcommand{\<}{\langle}
\renewcommand{\>}{\rangle}
\newcommand{\anno}[2]{\stackrel{\mathrm{#1}}{#2}}

\newcommand{\qqand}{\qquad\hbox{and}\qquad}

\tikzset{
  move/.style={>->},
  sub/.style={densely dashed,->},
  cover/.style={densely dashed,->>}
}
\draftdef{lsub}{\dashleftarrow}{\inlineArrow{sub, <-}}
\draftdef{rsub}{\dashrightarrow}{\inlineArrow{sub}}
\draftdef{lcover}{\looparrowleft}{\inlineArrow{sub}}
\draftdef{rcover}{\looparrowright}{\inlineArrow{sub, ->>}}
\draftdef{lmove}{\leftsquigarrow}{\inlineArrow{move}}
\draftdef{rmove}{\rightsquigarrow}{\inlineArrow{move}}

\newcommand{\arrow}{\diagArrow[m-]}
\def\sub{\arrow{sub}}
\def\cover{\arrow{cover}}
\def\move{\arrow{move}}

\renewcommand{\dot}{{{\raisebox{.2ex}{\scalebox{.3}{$\bullet$}}}}}

\def\morpair#1#2#3#4{\begin{tikzpicture}[baseline=(A.base)] \node[anchor=east]%
    (A) at (0,0) {$#1$}; \node[anchor=west] (B) at (2em,0) {$#2$};%
    \diagArrow{->, bend left}{A}{B}!{#3} \diagArrow{->, bend %
      right}{A}{B}!{#4}%
  \end{tikzpicture}}
\def\longmorpair#1#2#3#4#5{\begin{tikzpicture}[baseline=(A.base)] \node[anchor=east]%
    (A) at (0,0) {$#2$}; \node[anchor=west] (B) at (#1,0) {$#3$};%
    \diagArrow{->, bend left}{A}{B}!{#4} \diagArrow{->, bend %
      right}{A}{B}!{#5}%
  \end{tikzpicture}}

\def\mory#1#2#3#4#5#6{\begin{tikzpicture}[baseline=(A.base)] \node[anchor=east] (A) at (0,0) {$#1$}; \node[anchor=west] (B) at (2em,2ex) {$#2\sqcup \W(\iota)(#3)$}; \node[anchor=west] (C) at (2em,-2ex) {$#2\sqcup\W(\iota)(#4)$}; \diagArrow{<-, bend left}{A}{B.west}!{#5} \diagArrow{<-, bend  right}{A}{C.west}!{#6};%
  \end{tikzpicture}}
\newcommand{\bmory}[6]{\Big[\mory{#1}{#2}{#3}{#4}{#5}{#6}\Big]}

\newtheorem{maintheorem}{Theorem}

\theoremstyle{remark}
\newtheorem{example}[equation]{Example}
\newtheorem{non-example}[equation]{Counterexample}
\newtheorem{observation}[equation]{Observation}
\theoremstyle{definition}

\theoremstyle{plain}

\begin{document}
\title{On $K_1$ of an assembler}
\author{Inna Zakharevich}

\begin{abstract}
  This paper contains a construction of generators and partial relations for
  $K_1$ of a simplicial Waldhausen category where cofiber sequences split up to
  weak equivalence.  The main application of these generators and relations is
  to produce generators for $K_1$ of a (simplicial) assembler.
\end{abstract}

\maketitle


\section*{Introduction}

In this paper we develop tools with which we can analyze $K_1$ of an assembler.
Assemblers were introduced in \cite{Z-Kth-ass} as a formal framework with which
to analyze abstract scissors congruence relations using spectra instead of
groups; as this paper relies heavily on the notation and tools of the previous
paper, that paper should be read first.  The spectra used are obtained using
algebraic $K$-theory, which is notoriously difficult to compute.  In
\cite{Z-Kth-ass} there were no tools for computing higher homotopy groups of
assemblers.

In this paper we take inspiration from the author's previous work on polytope
complexes \cite{zakharevich10, zakharevich11} and Muro and Tonks' work on
identifying the $1$-type of a Waldhausen $K$-theory spectrum \cite{murotonks07,
  murotonks08} to construct a model for $K_1$ of an assembler.  The main theorem
of the paper is this:
\begin{maintheorem}
  Suppose that $\E_\dot$ is a nice simplicial Waldhausen category. Then every
  element of $K_1(\E_\dot)$ is represented by a diagram in $\E_0$ of the form
  \[B \sqcup d_1W \lwe^{f} A \rwe^g B \sqcup d_1V\] where $V,W$ are in $\E_1$ satisfy
  the additional condition that $d_0V_1 = d_0W_1$.
\end{maintheorem}
The technical conditions are listed in Proposition~\ref{prop:1typeE.}; the
theorem is proved in Theorem~\ref{thm:K1diag}.  The statement of the theorem
also includes several relations that these generators satisfy.

The main application of this theorem is this:
\begin{maintheorem}
  For any assembler $\C$, every element of 
    $K_1(\C)$ can be represented by a pair of morphisms
    \[\morpair{A}{B}{f}{g}\]
    in $\W(\C)$.  These satisfy the relations
    \[
    \big[\morpair{A}{B}{f}{f}\big] = 0,   \qquad 
    \big[\morpair{B}{C}{g_1}{g_2}\big] + \big[\morpair{A}{B}{f_1}{f_2}\big] =
    \big[\longmorpair{3em}{A}{C}{g_1f_1}{g_2f_2}\big]  \]
    and 
    \[\big[\morpair{A}{B}{f_1}{f_2}\big] + \big[\morpair{C}{D}{g_1}{g_2}\big] =
    \big[\begin{tikzpicture}[baseline] \node (A) at (0,0) {$A\amalg C$}; \node
      (B) at (5em,0) {$B\amalg D$}; \diagArrow{->, bend left}{A}{B}^{f_1\amalg g_1} \diagArrow{->,
        bend %
        right}{A}{B}_{f_2\amalg g_2}%
    \end{tikzpicture}\big]
    \]
\end{maintheorem}
This is proved in Corollary~\ref{cor:K1C}.  A more general statement for $K_1$
of the cofiber of a morphism of assemblers is given in
Corollary~\ref{cor:K1Cg}.  In addition, Theorem~\ref{thm:K1diag} gives a general
description of elements in $K_1$ of a (particularly nice type of) simplicial
Waldhausen category.

We give two applications of this theorem.  First, we compute $K_1$ of the
assembler of segments in a line, and show that it is isomorphic to the
abelianization of the group of interval exchange transformations.  For more
details, see Section~\ref{sec:iet}.  For the second, we describe how to compute
the differentials in the spectral sequence for scissors congruence groups
constructed in \cite[Section 5.2]{Z-Kth-ass}; see Section~\ref{sec:polytope}.  A
further application to the Grothendieck ring of varieties is developed in much
greater depth in \cite{Z-ass-var}.

This paper is organized as follows.  Section~\ref{sec:kthass} gives a quick
review of the definition of assemblers, as well as a definition of the
$K$-theory of an assembler using Waldhausen's construction.
Section~\ref{sec:S.} shows that this definition of the $K$-theory agrees with
the definition given in \cite{Z-Kth-ass}.  Section~\ref{sec:K1} develops the
model for $K_1$ of an assembler and proves the main theorem.
Sections~\ref{sec:iet} and \ref{sec:polytope} contain applications of these
theorems to interval exchange transformations and to the spectral sequence
associated to $n$-dimensional scissors congruence groups.
Section~\ref{app:murotonks} contains a proof of the main theorem of
Section~\ref{sec:K1}.

\section{The $K$-theory of a closed assembler} \lbl{sec:kthass}

In this section we recall the definition of a closed assembler, and construct
its $K$-theory spectrum as a Waldhausen category.  For more on assemblers, as
well as the original definition of the functor $K:\ASB \rto \Sp$, see
\cite[Section 1]{Z-Kth-ass}.

\begin{definition}
  In any category with an initial object $\initial$, we say that two morphisms
  $f:A\rto C$ and $g:B\rto C$ are \textsl{disjoint} if the pullback $A\times_C
  B$ exists and is equal to $\initial$.  A family $\{f_i:A_i \rto A\}_{i\in I}$
  is a \textsl{disjoint family} if for $i\neq i'$ the morphisms $f_i$ and
  $f_{i'}$ are disjoint.
\end{definition}

\begin{definition} 
  Let $\C$ be a Grothendieck site.  We denote the full subcategory of noninitial
  objects by $\C^\circ$.  We say that a family of maps $\{A_i\rto A\}_{i\in I}$
  is a \textsl{covering family} if it generates a covering sieve in the
  topology.

  A \textsl{closed assembler} is a small Grothendieck site $\C$ satisfying the
  following extra conditions:
  \begin{itemize}
  \item[(I)] $\C$ has an initial object $\initial$ and the empty family is a
    covering family of $\initial$.
  \item[(P)] $\C$ is closed under pullbacks.
  \item[(M)] All morphisms in $\C$ are monomorphisms.
  \end{itemize}
 
  A \textsl{morphism of closed assemblers} $F: \C\rto \D$ is a continuous (in
  the sense of Grothendieck topologies) functor which preserves pullbacks and
  the initial object.  We denote the category of closed assemblers and morphisms
  of closed assemblers by $c\ASB$.
\end{definition}

In order to construct scissors congruence categories associated to assemblers,
we need to be able to construct formal sums and decompositions of objects in an
assembler.  For this we need Grothendieck twists.

\begin{definition} \lbl{def:grtwist} The \textsl{Grothendieck twist of $\C$},
  written $\Tw(\C)$, is defined to be the category whose objects are tuples
  $\SCob{a}{i}$, where $I$ is a finite set and each $a_i$ is in $\ob\C$.  A morphism
  $\SCob{a}{i}\rto \SCob{b}{j}$ in $\Tw(\C)$ consists of a morphism of
  finite sets $f: I\rto J$, together with morphisms $f_i: A_i\rto B_{f(i)}$ for
  all $i\in I$.  In general we denote a morphism of $\Tw(\C)$ by a lower-case
  letter.  By an abuse of notation, we use the same letter to refer to the
  morphism's underlying map of sets, and subscripted versions of the letter to
  refer to the $\C$-components of the morphism (as we did above).
\end{definition}

The functor $\Tw(\C) \rto \FinSet$ given by forgetting the $\C$-components is a
fibration of categories.

\begin{proposition}
  Let $\C$ be any category and let $\D$ be a sieve in $\C$.  
  \begin{enumerate}
  \item If $\C$ has all pullbacks then so does $\Tw(\C \bs \D)$.
  \item If $\C$ has all finite connected colimits (equivalently: all pushouts
    and coequalizers) then $\Tw(\C)$ has all pushouts.  
  \item If $\C$ contains all pushouts then $\Tw(\C)$ is closed under pushouts
    along morphisms whose underlying maps of sets are injective.
  \end{enumerate}
\end{proposition}

\begin{proof}
  Part (1) is proved in \cite[Lemma 2.3]{zakharevich10}.  Parts (2) and (3) are
  proved in \cite[Lemma 2.4]{zakharevich10}.
\end{proof}

Let $\C$ be an assembler, and consider $\Tw(\C^\circ)$.  It contains three
distinguished types of morphisms.
\begin{definition}
  A morphism $f:\SCob{a}{i} \rto \SCob{b}{j}$ in $\Tw(\C^\circ)$ is a
  \begin{description}
  \item[sub-map] if for all $i,i'\in I$ such that $f(i) = f(i')$, the morphisms
    $f_i:A_i\rto B_{f(i)}$ and $f_{i'}:A_{i'} \rto B_{f(i')}$ are disjoint.
    Sub-maps are denoted $\rsub$.  The subcategory of sub-maps is
    denoted $\Tw(\C^\circ)^\Sub$.
  \item[covering sub-map] if it is a sub-map and for all $j\in J$ the family
    $\{f_i:A_i \rto B_j\}_{i\in f^{-1}(j)}$ is a covering family.  Covering
    sub-maps are denoted $\rcover$.  The subcategory of covering sub-maps
    is denoted $\W(\C)$
  \item[move] if for all $i\in I$, $f_i$ is an isomorphism.  Moves are
    denoted $\rmove$.
  \end{description}

  Note that $\W$ is a functor $c\ASB\rto \Cat$.
\end{definition}

\begin{proposition} \lbl{prop:Twprop}
  Let $\C$ be a closed assembler.  
  \begin{enumerate}
  \item $\Tw(\C^\circ)$ is closed under pullbacks and coproducts.  The pullback
    of a sub-map (resp. covering sub-map, move) is a sub-map (resp. covering
    sub-map, move).  The subcategory of sub-maps (resp. covering sub-maps,
    moves) is closed under coproducts.
  \item All morphisms in $\W(\C)$ are monomorphisms.
  \item $\W(\C)$ has all pullbacks.
  \item For any family of assemblers $\{\C_x\}_{x\in X}$, write $\bigvee_{x\in
      X}\C_x$ for the assembler whose objects are $\{\initial\}\cup
    \bigcup_{x\in X} \ob \C_x^\circ$ and where 
    \[\Hom(A,B) =
    \begin{cases}
      \Hom_{\C_x}(A,B) & \hbox{if }A,B\in \C_x, \\
      * & \hbox{if } A = \initial, \\
      \emptyset & \hbox{otherwise.}
    \end{cases}\] The topology is induced from the topologies of the $\C_x$.
    Then
    \[\Tw\big(\bigvee_{x\in X}\C_x\big) \simeq \bigoplus_{x\in X} \Tw(\C_x),\]
    where $\bigoplus \Tw(\C_x)$ is the full subcategory of $\prod \Tw(\C_x)$ where
    all but finitely many of the objects are the object indexed by the empty
    set.  The same holds with $\Tw$ replaced by $\W$.
  \end{enumerate}
\end{proposition}

\begin{proof}
  Part (1) is direction from the definitions.  Part (2) is proved in
  \cite[Proposition 1.10(1)]{Z-Kth-ass}.  Part (3) follows directly from axiom
  (P).  To prove part (4), observe that we have a functor
  \[P:\Tw(\bigvee_{x\in X} \C_x) \rto \prod_{x\in X} \Tw(\C_x)\] with projections
  \[F_x: \Tw(\bigvee_{x\in X} \C_x) \rto \Tw(\C_x)\] that send
  all $\C_y$ for $y\neq x$ to the initial object, and send $\C_x$ to itself via the
  identity.  As each object of $\Tw(\bigvee_{x\in X}\C_x)$ is indexed by a finite
  set it must land in $\bigoplus_{x\in X} \Tw(\C_x)$, so we see that $P$ is
  actually a functor $\Tw(\bigvee_{x\in X} \C_x) \rto \bigoplus_{x\in
    X}\Tw(\C_x)$.  To see that this is an equivalence, note that it is full and
  faithful and hits all objects indexed by disjoint indexing sets; since this is
  an equivalent subcategory, $P$ is also essentially surjective.  That this
  statement holds with $\Tw$ replaced by $\W$ is \cite[Proposition
  1.10(3)]{Z-Kth-ass}. 
\end{proof}

We are now ready to define the Waldhausen $K$-theory of an assembler.

\begin{definition} Let $\C$ be a closed assembler.  The category $\SC(\C)$ is
  defined to have $\ob\SC(\C) = \ob(\Tw(\C^\circ))$.  The morphisms of $\SC(\C)$
  are equivalence classes of diagrams in $\Tw(\C^\circ)$
  \[A \lsub^p A' \rmove^\sigma B,\] where $p$ is a sub-map and $\sigma$ is a
  move, and where two diagrams are considered equivalent if there is an
  isomorphism $\iota: A'_1 \rto A'_2\in \Tw(\C^\circ)$ which makes the following
  diagram commute:
  \begin{squisheddiagram}
    {& A_1' \\ A && B \\ & A_2' \\};
    \sub{1-2}{2-1}+{above=.5ex}{p_1} \sub{3-2}{2-1}+{below=.5ex}{p_2}
    \move{1-2}{2-3}+{above right=-.4ex}{\sigma_1} \move{3-2}{2-3}+{below right=-.4ex}{\sigma_2}
    \to{1-2}{3-2}^\iota
  \end{squisheddiagram}
  
  We say that a morphism $A \lsub^p A' \rmove^\sigma B$ is a
  \begin{description}
  \item[cofibration] if $p$ is a covering sub-map and the projection of $\sigma$
    to $\FinSet$ is injective, and a
  \item[weak equivalence] if $p$ is a covering sub-map and the projection of
    $\sigma$ to $\FinSet$ is bijective.
  \end{description}
  Note that any weak equivalence can be uniquely represented by a diagram where
  $A' = B$ and $\sigma=1_B$.  The composition of two morphisms $f: A\rto B$ and
  $g: B \rto C$ represented by
  \[A \lsub^p A' \rmove^\sigma B \qquad \hbox{and}\qquad B \lsub^q B' \rmove^\tau C\]
  is defined to be the morphism represented by the outside of the diagram
  \begin{squisheddiagram}
    { & & A'\times_B B' & & \\
      & A' &  & B' \\
      A && B && C \\}; \move{1-3}{2-4} \sub{1-3}{2-2} \sub{2-2}{3-1}_p
    \sub{2-4}{3-3}_q \move{2-2}{3-3}^\sigma \move{2-4}{3-5}^\tau
  \end{squisheddiagram}
  The left-hand side of the pullback square is a sub-map and the right-hand side
  is a move by Proposition~\ref{prop:Twprop}(1).
\end{definition}

\begin{remark}
  The subcategory of weak equivalences of $\SC(\C)$ is isomorphic to
  $\W(\C)^\op$.
\end{remark}

\begin{theorem} \lbl{thm:Wald} $\SC$ is a functor $c\ASB\rto \mathbf{WaldCat}$.
  Every Waldhausen category in the image of $\SC$ has a canonical splitting (up
  to weak equivalence) for every cofibration sequence.
\end{theorem}

We omit the proof of the theorem, as it is essentially the same as the proof
of \cite[Theorem 4.2]{zakharevich10}; for more information on translating
between polytope complexes and assemblers, see \cite[Example 2.10]{Z-Kth-ass} and
Appendix~\ref{app:technical}.  It is important to note that this theorem only
holds for closed assemblers, not all assemblers, as pullbacks in $\C$ are
necessary to construct pushouts inside $\SC(\C)$.



\begin{definition}
  For any closed assembler $\C$, we define $K^W(\C)$, the \textsl{Waldhausen
    $K$-theory spectrum} of $\C$ to be $K(\SC(\C))$.  We write $K_i(\C)$ for
  $\pi_iK^W(\C)$.
\end{definition}

We see in Theorem~\ref{thm:2models} that this is consistent with the
definition of $K_i(\C)$ in \cite{Z-Kth-ass}.

There is one aspect of the Waldhausen categories $\SC(\C)$ that it is important
to mention here.  Not all categories $\SC(\C)$ satisfy the Saturation Axiom; in
other words, the subcategory of weak equivalences $w\SC(\C)$ does not
necessarily satisfy two-of-three.  The following proposition (originally
appearing as \cite[Lemma 6.9]{zakharevich10}) highlights the cases in which
saturation is satisfied:

\begin{proposition} \lbl{prop:saturation} For any two composable morphisms $f$
  and $g$ in $\SC(\C)$, if $gf$ and $f$ are weak equivalences then so is $g$.
  If $\C$ satisfies the extra condition
  \begin{itemize}
  \item[(G)] The empty family is not a covering family for any noninitial object
    of $\C$.  Given a family $\mathcal{A}=\{X_\alpha \rto X\}_{\alpha\in A}$
    and covering families $\{X_{\alpha\beta} \rto X_\alpha\}_{\beta\in
      B_\alpha}$, if the refined family \[\{X_{\alpha\beta} \rto
    X\}_{(\alpha,\beta)\in \coprod_{\alpha\in A}B_\alpha}\] is a covering family
    then so is $\mathcal{A}$.
  \end{itemize}
  then if $gf$ and $g$ are weak equivalences, then so is $f$. 
\end{proposition}

In other words, if (G) is satisfied then $\SC(\C)$ satisfies the Saturation
Axiom.  Condition (G) is satisfied in all of the examples we have considered so
far; however, in future work it will be necessary to consider assemblers which
do not satisfy this condition.

To finish this section, we recall the notion of a simplicial assembler and its
Waldhausen $K$-theory.

\begin{definition} \lbl{def:simpass} A \textsl{simplicial closed assembler} is a
  functor $\C_\dot:\Delta^\op \rto c\ASB$.  A \textsl{morphism of simplicial closed
    assemblers} is a natural transformation of functors.  We define the
  Waldhausen $K$-theory spectrum by
  \[K^W(\C_\dot) = \hocolim_{[n]\in\Delta^\op} K^W(\C_n).\]
\end{definition}

\begin{example} 
  For any pointed set $S$ and any closed assembler $\C$, let 
  \[S \smash \C = \bigvee_{S \bs \{*\}} \C.\] This is functorial in $S$: given a
  map $f:S \rto T$ we get a map $S \smash \C \rto T \smash \C$ in the following
  manner.  For all $s\in S$ such that $f(s) \neq *$ the copy of $\C$ indexed by
  $s$ is mapped via the identity to the copy of $\C$ indexed by $f(s)$.  When
  $f(s) = *$ the copy of $\C$ indexed by $s$ is mapped to the initial object.

  Let $X_\dot$ be a pointed simplicial set.  Then for any
  closed assembler $\C$, the simplicial closed assembler $X_\dot \smash \C$ is
  defined by
  \[(X_\dot\smash\C)_k = X_k \smash \C.\]
  The simplicial maps are induced from the simplicial maps on $X_\dot$.
\end{example}

\section{A comparison of $K$ and $K^W$}  \lbl{sec:S.}

We now have two different definitions of $K_i(\C)$ for a closed assembler $\C$:
the definition from \cite{Z-Kth-ass} and the definition given here.  Our goal in
this section is to prove that they are equivalent.

\begin{theorem} \lbl{thm:2models} There is a natural transformation $\eta: K
  \Rto K^W$ such that for every closed assembler $\C$ and $k\geq 1$, the $k$-th
  component of $\eta_\C$ is a weak equivalence $K(\C)_k \rto K^W(\C)_k$.
\end{theorem}

If we are given two $k$-simplicial spaces then a level equivalence gives an
equivalence on geometric realization, so we do a levelwise analysis of the
categories that assemble to construct $K(\C)$ and $K^W(\C)$.  First we recall
two results from \cite{zakharevich11}:

\begin{lemma}[{\cite[Lemma 5.3]{zakharevich11}}] \lbl{lem:5.3}
  Let $L_n\SC(\C)$ be the full subcategory of $S_n\SC(\C)$ containing all
  objects 
  \[A_1\racyccofib A_2 \racyccofib \cdots \racyccofib A_n\] in which all
  morphisms are layered (in the sense of \cite[Definition 5.1]{zakharevich11}).
  Then $L_n\SC(\C)$ is a Waldhausen category and the inclusion $L_n\SC(\C)
  \rcofib S_n\SC(\C)$ induces the identity morphism on $K$-theory.
\end{lemma}

\begin{proposition}[{\cite[Proposition 5.5]{zakharevich11}}] \lbl{prop:5.5}
  Let $W_n\SC(\C)$ be the full subcategory of $S_n\SC(\C)$ containing all
  objects 
  \[A_1\racyccofib A_2 \racyccofib \cdots \racyccofib A_n.\] Then $W_n\SC(\C)$
  is a Waldhausen category and we have an exact equivalence of Waldhausen
  categories
  \[\mathrm{St}: L_n\SC(\C) \rlto \bigoplus_{i=1}^n W_i\SC(\C) \,:\! CP.\]
\end{proposition}

The proofs from \cite{zakharevich11} are stated for the case when $\C$ is a
polytope complex, but they work almost verbatim for the case when $\C$ is an
assembler as well, so we do not reproduce them here.

\begin{lemma} \lbl{lem:pncnadj}
  Let $p_n:W_n\SC(\C) \rto \SC(\C)$ be the functor which takes $A_1\racyccofib
  \cdots \racyccofib A_n$ to $A_n$, and let $c_n: \SC(\C) \rto W_n\SC(\C)$ be
  the functor that takes $A$ to the diagram $A \req \cdots \req A$.  Then $(p_n
  \dashv c_n)$ is an exact adjunction.
\end{lemma}

\begin{proof}
  $c_n$ is exact by definition, and $p_n$ is also exact once we notice that all
  cofibrations, weak equivalences, and pushouts are levelwise in $W_n\SC(\C)$.
  Thus all we need to check is that the functors are adjoints and that the unit
  and counit are natural weak equivalences.  As $p_nc_n = 1_{\SC(\C)}$ the
  counit is the identity.  The unit is represented by the diagram
  \begin{diagram}
    {A_1 & A_2 & \cdots & A_n \\ 
      A_n & A_n & \cdots & A_n \\};
    \acyccofib{1-1}{1-2}^{f_1} \acyccofib{1-2}{1-3}^{f_2}
    \acyccofib{1-3}{1-4}^{f_{n-1}}
    \eq{2-1}{2-2} \eq{2-2}{2-3} \eq{2-3}{2-4}
    \acyccofib{1-1}{2-1}_{g_1} \acyccofib{1-2}{2-2}_{g_2} 
    \eq{1-4}{2-4}
  \end{diagram}
  where $g_i = f_{n-1}\circ\cdots\circ f_i$.  Both the unit and counit are
  levelwise equivalences, and thus are weak equivalences in $W_n\SC(\C)$ and
  $\SC(\C)$, respectively.  Checking that they satisfy the relations to be the
  unit and counit of an adjunction is direct from the definition.
\end{proof}

\begin{definition}
  For any nonnegative integer $n$, let $n^+$ be the pointed set
  $\{*,1,\ldots,n\}$ with $*$ as its basepoint.  Let $S^1$ be the pointed
  simplicial set with $(S^1)_n = n^+$, and let $S^k$ be the smash product of
  $S^1$ with itself $k$ times.
  
  For any $k$-tuple of nonnegative integers $n_1,\ldots,n_k$ write
  \[S^k_{n_1,\ldots,n_k} = (S^1)_{n_1}\smashes (S^1)_{n_k}.\] 
  For any Waldhausen category $\E$ write 
  \[S_{n_1,\ldots,n_k}\E = S_{n_1}\cdots S_{n_k}\E.\]
\end{definition}

\begin{proposition} \lbl{prop:levelex}
  For any $k$-tuple of nonnegative integers $n_1,\ldots,n_k$, there is an exact
  functor of Waldhausen categories, natural in $\C$, 
  \[C_{n_1,\ldots,n_k}:\SC(n_1^+\smashes n_k^+\smash \C) \rto
  S_{n_1,\ldots,n_k}\SC(\C)\] which induces a homotopy equivalence after
  applying $|w\cdot|$.  
\end{proposition}

\begin{proof}
  If any of the $n_i$ is equal to zero, then both sides are just the trivial
  Waldhausen category, so the proposition clearly holds.  Thus we can assume
  that all $n_i$ are positive.

  We prove the proposition by induction on $k$.  In the base case $k=1$ we have
  an exact functor
  \[C_{n}:\SC(n^+\smash \C) \rto^\simeq \bigoplus_{i=1}^n \SC(\C)
  \rto^{\bigoplus_{i=1}^n c_{n-i}} \bigoplus_{i=1}^n W_i\SC(\C) \rto^{CP}
  L_n\SC(\C) \rto S_n\SC(\C),\] where the first three functors are exact
  equivalences, and the last one induces the identity morphism on $K$-theory.
  In particular, the composition induces a homotopy equivalence after applying
  $|wS_{n_1}\cdots S_{n_\ell}\cdot|$ for all tuples $(n_1,\ldots,n_\ell)$; the
  proposition follows by setting $\ell=0$; we mention the more general statement
  as it is necessary in the proof of the inductive step.

  Now suppose that we know the proposition holds up to $k-1$.  Letting $\D =
  n_k^+\smash \C$, by the induction hypothesis we have an exact functor
  \[C_{n_1,\ldots,n_{k-1}}:\SC(n_1^+\smashes n_{k-1}^+\smash \D) \rto
  S_{n_1,\ldots,n_{k-1}}\SC(\D)\] which is a homotopy equivalence after applying
  $|w\cdot|$.  By the base case we have an exact functor $\SC(\D) \rto
  S_{n_k}\SC(\C)$ which induces a homotopy equivalence after applying
  $|wS_{n_1}\cdots S_{n_{k-1}}\cdot|$.  Thus the composite
  \[C_{n_1,\ldots,n_k}:\SC(n_1^+\smashes n_{k-1}^+\smash \D)
  \rto^{C_{n_1,\ldots,n_{k-1}}} S_{n_1,\ldots,n_{k-1}}\SC(\D)
  \rto^{S_{n_1,\ldots,n_{k-1}}C_{n_k}} S_{n_1,\ldots,n_k}\SC(\C)\] is also
  a homotopy equivalence after applying $|w\cdot|$.  As every step of this
  construction is natural in $\C$, we see that the functors are also natural in
  $\C$.
\end{proof}

We show that the above functors $C_{n_1,\ldots,n_k}$ properly assemble into
$k$-simplicial functors, so that we have induced maps
\[\big|w\SC(S^k\smash \C)\big| \rwe \big|wS_\dot\cdots S_\dot \SC(\C)\big|\]
which are weak equivalences because they are levelwise weak equivalences.  To
check that this works, we just need to check that these maps commute properly
with the simplicial structure maps.

\begin{lemma} \lbl{lem:simpassembly}
  For all $1\leq j\leq k$ and for all $0\leq i \leq n_j$, the following diagram
  commutes: 
  \begin{diagram}[7em]
    {\SC(S^k_{n_1,\ldots,n_k} \smash \C) & S_{n_1,\ldots,n_k} \SC(\C) \\
      \SC(S^k_{n_1,\ldots,n_j-1,\ldots,n_k} \smash \C) &
      S_{n_1,\ldots,n_j-1,\ldots,n_k} \SC(\C) \\};
    \to{1-1}{1-2}^{C_{n_1,\ldots,n_k}}
    \to{2-1}{2-2}^{C_{n_1,\ldots,n_j-1,\ldots,n_k}}
    \to{1-1}{2-1}_{1,\ldots,d_i,\ldots,1} \to{1-2}{2-2}^{1,\ldots,d_i,\ldots,1}
  \end{diagram}
\end{lemma}

\begin{proof}
  We prove this by induction on $k$.  First consider the base case $k=1$. We
  want to show that the following diagram commutes:
  \begin{diagram}
    {\SC((S^1)_n \smash \C) & S_n \SC(\C) \\ 
      \SC((S^1)_{n-1} \smash \C) & S_{n-1} \SC(\C) \\};
    \arrowsquare{C_n}{d_i}{d_i}{C_{n-1}}
  \end{diagram}
  This is straightforward from the definition of $C_n$ and $C_{n-1}$.

  Now suppose that we know the lemma for $k-1$, and consider it for $k$.  Let
  $\D = n_k^+\smash \C$. If $j<k$ the construction of $C_{n_1,\ldots,n_k}$ lets
  us rewrite the above diagram in the following manner:
  \begin{diagram}
    { \SC(S^{k-1}_{n_1,\ldots,n_{k-1}} \smash \D) &
      S_{n_1,\ldots,n_{k-1}}\SC(\D) &
      S_{n_1,\ldots,n_k}\SC(\C) \\
      \SC(S^{k-1}_{n_1,\ldots,n_j-1,\ldots,n_k}\smash \D) &
      S_{n_1,\ldots,n_j-1,\ldots,n_{k-1}}\SC(\D) &
      S_{n_1,\ldots,n_j-1,\ldots,n_k}\SC(\C) \\};
    \arrowsquare{C_{\cdots}}{1,\ldots,d_i,\ldots,1}{1,\ldots,d_i,\ldots,1}{C_{\cdots}}
    \to{1-2}{1-3}
    \to{2-2}{2-3}
    \to{1-3}{2-3}^{1,\ldots,d_i,\ldots,1}
  \end{diagram}
  The right-hand square commutes because $C_{n_k}:\SC(\D) \rto S_n\SC(\C)$ is an
  exact functor and thus induces a natural transformation of simplicial
  Waldhausen categories
  \[S_\dot\cdots S_\dot \SC(\D) \rto S_\dot\cdots S_\dot \SC(\C).\]
  The left-hand square commutes by the induction hypothesis.

  Now suppose that $j=k$ and let $\D'= (n_k-1)^+\smash \C$.  Then we can rewrite
  the square in the statement of the lemma as
  \begin{diagram}
    {\SC(S^{k-1}_{n_1,\ldots,n_{k-1}} \smash \D) & S_{n_1,\ldots,n_{k-1}}\SC(\D)
      &
      S_{n_1,\ldots,n_k}\SC(\C) \\
      \SC(S^{k-1}_{n_1,\ldots,n_{k-1}}\smash \D') &
      S_{n_1,\ldots,n_{k-1}}\SC(\D') & S_{n_1,\ldots,n_k-1}\SC(\C) \\};
    \arrowsquare{C_{\cdots}}{1,\ldots,d_i}{1,\ldots,d_i}{C_{\cdots}}
    \to{1-2}{1-3} \to{2-2}{2-3} \to{1-3}{2-3}^{1,\ldots,d_i}
  \end{diagram}
  Here the left-hand square commutes because the construction of $C_{\cdots}$ is
  natural in $\C$, and the vertical maps are induced from a morphism of
  assemblers $\D\rto \D'$.  The right-hand square is $S_{n_1}\cdots S_{n_{k-1}}$
  applied to the base case, and thus it also commutes.
\end{proof}

Lemma~\ref{lem:simpassembly} shows that we can assemble the $C$'s into morphisms
between realizations of $k$-simplicial sets.  Thus $|C_{\cdots}|$ gives a weak
equivalence $|w\SC(S^k\smash \C)| \rwe |wS_\dot\cdots S_\dot \SC(\C)|$ for all
$k$.  In order to prove Theorem~\ref{thm:2models} it suffices to show that these
maps assemble into a morphism of spectra.

\begin{lemma} \lbl{lem:spectralassembly}
  The following diagram commutes for all $k$:
  \begin{diagram}[5em]
    {S^1 \smash |w\SC(S^k\smash \C)| & S^1 \smash |wS_\dot^{(k)} \SC(\C)|
      \\
      {} |w\SC(S^{k+1}\smash \C)| & {} |wS_\dot^{(k+1)} \SC(\C)| \\};
    \arrowsquare{S^1\smash |C_{\cdots}|}{}{}{|C_{\cdots}|}
  \end{diagram}
\end{lemma}

\begin{proof}
  This follows directly from the definition of the spectral structure maps.  For
  an explicit description of these for the $S_\dot$ construction, see
  \cite[Section 6]{Z-enrich-wald}.
\end{proof}

We now have all of the pieces of the proof of the main theorem:

\begin{proof}[Proof of Theorem~\ref{thm:2models}]
  We define the map $\eta_\C:K(\C) \rto K^W(\C)$ componentwise by
  \[|C_{\cdots}|: |\W(S^k\smash \C)| = |w\SC(S^k \smash \C)| \rwe
  |wS_\dot^{(k)}\SC(\C)|,\] where the equality follows because $\W(\D) =
  (w\SC(\D))^\op$ for all closed assemblers $\D$.  By
  Proposition~\ref{prop:levelex} and Lemma~\ref{lem:simpassembly} these maps are
  well-defined weak equivalences.  We also know, by
  Lemma~\ref{lem:spectralassembly} that these assemble into well-defined
  morphisms of spectra.  That $\eta$ is a natural transformation follows from
  Proposition~\ref{prop:levelex}.
\end{proof}

Applying \cite[Theorem C]{Z-Kth-ass} we conclude the following:

\begin{corollary}
  For any simplicial closed assembler $\C_\dot$, $K(\C_\dot) \rto K^W(\C_\dot)$
  is a weak equivalence.  For any morphism of simplicial closed assemblers
  $g:\D_\dot \rto \C_\dot$,
  \[K^W(\D_\dot) \rto K^W(\C_\dot) \rto K^W((\C/g)_\dot)\]
  is a cofiber sequence.
\end{corollary}

\section{$K_1$ of an assembler} \lbl{sec:K1}

In this section we give a combinatorial description of $K_1$ of an assembler,
similar to the description of $K_0$ in \cite[Theorem A]{Z-Kth-ass}.  As the
structure of $K_1$ is much more complicated than the structure of $K_0$, the
associated description must also be more complicated.  However, it is possible
to make the description simple enough that it can be used to compute the
boundary homomorphism $K_1((\C/g)_\dot) \rto K_0(\D)$.  In order to write down
this description we must first establish some terminology.

The results in this section are inspired by the work of Muro and Tonks in
\cite{murotonks07, murotonks08}.  

\begin{definition}
  A \textsl{stable quadratic module} $C_*$ consists of the data of a pair of
  groups $C_0$ and $C_1$ together with a homomorphism $\partial: C_1 \rto \C_0$
  and a bilinear form $\<\cdot,\cdot\>:C_0^\ab \otimes C_0^\ab \rto C_1$
  satisfying the following relations:
  \begin{itemize}
  \item[(SQ1)] for $c_0,d_0$ in $C_0$, $\partial\<c_0,d_0\> = [d_0,c_0]$,
  \item[(SQ2)] for $c_1,d_1$ in $C_1$, $\<\partial c_1, \partial d_1\> =
    [d_1,c_1]$, and
  \item[(SQ3)] for $c_0,d_0$ in $C_0$, $\<c_0,d_0\>\<d_0,c_0\> = 0$.
  \end{itemize}
  Here, $[x,y] = x^{-1}y^{-1}xy$.
\end{definition}

Stable quadratic modules were first introduced in \cite{baues91}.  However, we
use the convention for $\<\cdot,\cdot\>$ used in \cite{murotonks07,murotonks08}
and switch to multiplicative notation to emphasize that $C_0$ and $C_1$ are not
necessarily abelian.  In any stable quadratic module there is a right action
of $C_0$ on $C_1$ defined by
\[c_1^{c_0} = c_1\<c_0,\partial c_1\>.\] In addition, directly from the axioms
it follows that $\im \<\cdot,\cdot\>$ and $\ker \partial$ are central in $C_1$.
In this section we will be constructing stable quadratic modules via generators
and relations.  Note that we will not be able to write down a complete set of
relations induced by the stable quadratic module structure; to see a detailed
discussion of stable quadratic modules via generators and relations, see
\cite[Appendix A]{murotonks07}.

We write
\[\tau_{A_0,B_0} \anno{def}{=} B_0 \sqcup A_0 \rwe A_0 \sqcup B_0.\]

\begin{definition}
  Let $\E_\dot$ be a simplicial Waldhausen category.  We define the stable
  quadratic module $\curD_*\E_\dot$ to be the stable quadratic module generated
  by 
  \begin{itemize}
  \item a generator $[A_0]$ for every object $A_0$ in $\E_0$ in degree $0$, 
  \item a generator $[A_1]$ for every object $A_1$ in $\E_1$ in degree $1$, and 
  \item a generator $[A_0 \rwe B_0]$ for every weak equivalence in $\E_0$ in
    degree $1$.
  \end{itemize}
  On the generators, we define
  \[\partial[A_1] = [d_0A_1]^{-1}[d_1A_1] \qquad\hbox{and}\qquad \partial[A_0
  \rwe B_0] = [B_0]^{-1}[A_0]\]
  and
  \[\<[A_0],[B_0]\> = [\tau_{A_0,B_0}].\]
  These satisfy the following relations:
  \begin{itemize}
  \item[(A1)] For all objects $A_0$ in $\E_0$, $[s_0A_0] = [A_0 \req A_0] = 1$.
  \item[(A2)] For all objects $A_0,B_0$ in $\E_0$, $[A_0 \sqcup B_0] = [A_0][B_0]$.
  \item[(A3)] For every composable pair of weak equivalences $A_0 \rwe^f B_0
    \rwe^g C_0$, 
    \[ [A_0 \rwe^{gf} C_0] = [B_0 \rwe^g C_0][A_0 \rwe^f B_0].\]
  \item[(A4)] For every object $A_2$ in $\E_2$, $[d_1A_2] = [d_0A_2][d_2A_2]$.
  \item[(A5)] For every weak equivalence $A_1 \rwe B_1$ in $\E_1$, 
    \[[B_1][d_1A_1 \rwe d_1B_1] = [d_0A_1 \rwe d_0B_1][A_1].\]
  \item[(A6)] For all objects $C_0$ and weak equivalences $A_0 \rwe^f B_0$ in
    $\E_0$, 
    \[[C_0 \sqcup A_0 \rwe^{1_{C_0}\sqcup f} C_0 \sqcup B_0] = [A_0 \rwe^f
    B_0].\]
    For any object $A_1$ in $\E_1$, 
    \[[A_1] = [s_0C_0 \sqcup A_1].\]
  \item[(A7)] For any two objects $A_1,B_1$ in $\E_1$,
    \[[A_1 \sqcup B_1] = [B_1][A_1\sqcup s_0d_1B_1].\]
  \end{itemize}
\end{definition}

\begin{lemma} \lbl{lem:Dfunc}
  $\curD_*$ is a functor from the category of Waldhausen categories and exact
  functors to the category of stable quadratic modules.  Exact equivalences of
  categories induce homotopy equivalences of stable quadratic modules.
\end{lemma}

\begin{proof}
  This follows by the same proof as \cite[Theorem 3.2]{murotonks07}.
\end{proof}

\begin{proposition} \lbl{prop:1typeE.}  Let $\E_\dot$ be a simplicial Waldhausen
  category satisfying the following additional conditions:
  \begin{enumerate}
  \item $\E_0$, $\E_1$ and $\E_2$ have strictly associative and unital
    coproducts which are compatible with the simplicial structure maps.
  \item There exists a set $S_0$ of objects of $\E_0$ which freely generates the
    monoid of objects of $\E_0$.
  \item All cofiber sequences in $\E_0$ split compatibly up to weak equivalence,
    in the sense that for every cofiber sequence $A \rcofib B \rfib B/A$ there
    exists a morphism $\alpha:B/A \rwe B$ which fits into a commutative diagram
    \begin{diagram-fixed}
      {A & B/A \sqcup A & B/A \\ A & B & B/A \\};
      \eq{1-1}{2-1} \eq{1-3}{2-3} \we{1-2}{2-2}_\alpha
      \cofib{1-1}{1-2} \cofib{2-1}{2-2}
      \fib{1-2}{1-3} \fib{2-2}{2-3}
    \end{diagram-fixed}
    such that for every diagram
    \begin{diagram}
      {  & & C/B \\
        & B/A & C/A \\
        A & B & C \\};
      \cofib{3-1}{3-2} \cofib{3-2}{3-3} \cofib{2-2}{2-3}
      \fib{3-2}{2-2} \fib{3-3}{2-3} \fib{2-3}{1-3}
    \end{diagram}
    the square
    \begin{diagram}[4em]
      { C/B \sqcup B/A \sqcup A & C/B \sqcup B \\
        C/A \sqcup A & C \\};
      \arrowsquare{1_{C/B} \sqcup \alpha}{\alpha\sqcup
        1_A}{\alpha}{\alpha} 
    \end{diagram}
    commutes.
  \end{enumerate}
  Then $\curD_*\E_\dot$ encodes the $1$-type of $K(\E_\dot)$, in the sense that 
  \[K_0(\E_\dot) \cong \coker \partial \qquad\hbox{and}\qquad K_1(\E_\dot) \cong
  \ker \partial,\]
  and the first Postnikov invariant can be computed from $\curD_*\E_\dot$.
\end{proposition}

We postpone the proof of this proposition to Section~\ref{app:murotonks}.

\begin{remark}
  Conditions (1) and (2) in the proposition are present in order to simplify
  computations.  A levelwise construction analogous to \cite[Proposition
  4.3]{murotonks08} should work on any simplicial Waldhausen category to produce
  one where (1) and (2) hold.  However, the third condition is nontrivial and
  does not hold in most Waldhausen categories.
\end{remark}

Now we prove a result analogous to that of \cite[Theorem 2.2]{murotonks08},
which constructs a diagram such that every element of $K_1$ of a Waldhausen
category is represented by a diagram of that shape.

\begin{definition}
  Let $A_0,B_0$ be objects in $\E_0$ and $V_1$ be an object in $\E_1$.  Suppose
  that there exists a weak equivalence $f:A_0\rwe B_0  \sqcup d_1V_1 $ in $\E_0$.
  Then we write $\{f_{V_1}:A_0 \rwe B_0\}$ for the element in $\curD_1\E_\dot$
  represented by
  \[[V_1][A_0 \rto^f B_0 \sqcup d_1V_1].\]
  For objects $V_1,W_1$ in $\E_1$ such that $d_0V_1 = d_0W_1$ we use the shorthand
  $\{f_{V_1},g_{W_1}:A_0 \rwe B_0\}$
  for the element
  \[\{f_{V_1}:A_0 \rwe B_0\}^{-1}\{g_{W_1}:A_0 \rwe B_0\}.\]
\end{definition}

Note that
\[ \partial(\{f_{V_1},g_{W_1}:A_0 \rwe B_0\}) = 0.\] Thus $\{f_{V_1},g_{W_1}:A_0
\rwe B_0\}$ represents an element of $K_1(\E_\dot)$, and is central in
$\curD_1\E_\dot$.

The rest of this section relies extensively on relations (B1)-(B6) proved in
Lemma~\ref{lem:Brels}.  For the reader interested in the technical details, we
recommend consulting that lemma first; we postpone stating and proving this
lemma, as it is technical and does not affect the overall understanding of the
results in this section.

\begin{lemma} \lbl{lem:addingcurly}
  Suppose that $V_1,W_1,X_1,Y_1$ are objects in $\E_1$ such that $d_0V_1 =
  d_0W_1$ and $d_0X_1 = d_0Y_1$.  
  Then 
  \[\begin{array}{l}
    \{f_{V_1},g_{W_1}:A_0 \rwe B_0\}\{r_{X_1},s_{Y_1}:C_0 \rwe
    D_0\} = \\\qquad
    \{(f\sqcup r)_{V_1 \sqcup X_1}, (g\sqcup s)_{W_1\sqcup Y_1}:A_0 \sqcup C_0
    \rwe B_0 \sqcup D_0\}.\end{array}\]
  If in addition $D_0 = A_0$ then this is also equal to
  \[\{(fr)_{X_1\sqcup V_1}, (gs)_{Y_1\sqcup W_1}: C_0 \rwe B_0\}.\]
\end{lemma}

\begin{proof}
Observe that since
  $[\tau_{A,B}]$ is central in $\curD_1\E_\dot$, by (SQ2) all commutators are also
  central.  Thus we compute that 
  \[
  \begin{array}{l}
    \{g_{W_1}:A_0 \rwe B_0\}\{s_{Y_1}:C_0 \rwe D_0\} = \\
    \qquad= [W_1][Y_1][A_0 \rwe B_0 \sqcup d_1W_1][C_0 \rwe D_0 \sqcup
    d_1Y_1]\Big[[A_0 \rwe B_0 \sqcup d_1W_1],[Y_1]\Big] \\
    \qquad\anno{(A3)}= [W_1\sqcup Y_1][D_0 \sqcup d_1W_1 \sqcup d_1Y_1 \rwe
    d_1W_1 \sqcup D_0 \sqcup d_1Y_1][A_0 \sqcup C_0 \rwe B_0 \sqcup D_0 \sqcup
    d_1W_1 \sqcup d_1Y_1] \\
    \qquad\qquad[\tau_{d_0Y_1,d_0W_1}][\tau_{d_1W_1,d_0Y_1}][\tau_{D_0 \sqcup
      d_1Y_1,B_0 \sqcup d_1W_1}][\tau_{A_0,D_0\sqcup d_1Y_1}] \Big[[A_0 \rwe B_0
    \sqcup d_1W_1],[Y_1]\Big] \\
    \qquad\anno{(B6),(B2)}= \{(g\sqcup s)_{W_1\sqcup Y_1}: A_0 \sqcup C_0 \rwe B_0
    \sqcup D_0\} \alpha \\
    \qquad\qquad [\tau_{d_0Y_1,d_0W_1}][\tau_{d_1W_1,d_0Y_1}]
    [\tau_{d_1Y_1,B_0}][\tau_{d_1Y_1,d_1W_1}][\tau_{A_0,d_1Y_1}] \Big[[A_0 \rwe
    B_0 \sqcup d_1W_1],[Y_1]\Big] 
    \\
    \qquad= \{(g\sqcup s)_{W_1\sqcup Y_1}: A_0 \sqcup C_0 \rwe B_0
    \sqcup D_0\} \alpha [\tau_{d_0Y_1,d_0W_1}] [\tau_{d_0Y_1,B_0}][\tau_{A_0,d_0Y_1}].    
  \end{array}
  \]
  where $\alpha =
  [\tau_{D_0,B_0}][\tau_{A_0,D_0}]$.  
   Analogously we compute that
  \[\begin{array}{l}
    \{f_{V_1}:A_0 \rwe B_0\}\{r_{X_1}:C_0 \rwe D_0\} = \\
    \qquad= \{(f\sqcup r)_{V_1 \sqcup X_1}: A_0 \sqcup C_0 \rwe B_0 \sqcup
    D_0\}\alpha[\tau_{d_0X_1,d_0V_1}][\tau_{d_0X_1,B_0}][\tau_{A_0,d_0X_1}].
  \end{array}\]

  Note that 
  \[\begin{array}{l}
    \{f_{V_1},g_{W_1}:A_0 \rwe B_0\}\{r_{X_1},s_{Y_1}:C_0 \rwe D_0\}\\
    \qquad= \{r_{X_1}:C_0 \rwe D_0\}^{-1}\{f_{V_1}:A_0 \rwe B_0\}^{-1}\{g_{W_1}:A_0 \rwe
    B_0\}\{s_{Y_1}:C_0 \rwe D_0\},
  \end{array}\]
  since $\{f_{V_1},g_{W_1}:A_0 \rwe B_0\}$ is central.    Thus the difference between the left hand side and the right-hand side of the
  desired equality is 
  \[([\tau_{d_0X_1,d_0V_1}][\tau_{d_0X_1,B_0}][\tau_{A_0,d_0X_1}])^{-1}[\tau_{d_0Y_1,d_0W_1}]
  [\tau_{d_0Y_1,B_0}][\tau_{A_0,d_0Y_1}].\]
  However, since $d_0V_1 = d_0W_1$ and $d_0X_1 = d_0Y_1$ this is equal to $1$,
  and the desired equality follows.

  The second formula follows by a similar analysis and the observation that
  \[\{f_{V_1}: A_0 \rwe B_0\}\{r_{X_1}:C_0 \rwe A_0\} = \{(fr)_{X_1\sqcup
    V_1}:C_0 \rwe B_0\}[\tau_{d_0V_X,B_0}][\tau_{A_0,d_0X_1}].\] Since $d_0X_1
  = d_0Y_1$ the $\tau$'s cancel out, and we are left with the desired
  relation.
\end{proof}

The following theorem is the main technical result of the paper.  It is
important to note that we do not know that the stated relations are {\em all} of
the relations satisfied by the generators.  As observed in \cite{murotonks08},
the structure of the stable quadratic module imposes additional relations
between generators than the ones listed above, and this may lead to extra
relations between the generators of $K_1$.  However, this theorem is sufficient
for several applications; see Sections~\ref{sec:iet} and \ref{sec:polytope} and
\cite{Z-ass-var}.

\begin{theorem} \lbl{thm:K1diag} Suppose that $\E_\dot$ satisfies the
  conditions in Proposition~\ref{prop:1typeE.}.  Every element of $K_1(\E_\dot)$
  is represented by some
  \[\{f_{V_1},g_{W_1}: A_0 \rwe B_0\}\]
  where $V_1$ and $W_1$ satisfy the extra condition that $d_0V_1 = d_0W_1$.
  These satisfy the relations
  \begin{eqnarray*}
    && \{f_{V_1},f_{V_1}: A_0 \rwe B_0\} = 0 \\
    && \{f_{V_1},g_{W_1}: A_0 \rwe B_0\}\{r_{X_1},s_{Y_1}:C_0 \rwe D_0\}\\
    &&\qquad=
    \{(f\sqcup r)_{V_1\sqcup X_1}, (g\sqcup s)_{W_1\sqcup Y_1}: A_0 \sqcup C_0
    \rwe B_0 \sqcup D_0\}.\\
    && \{f'_{V_1'},g'_{W_1'}:B_0 \rwe C_0\}\{f_{V_1},g_{W_1}:A_0 \rwe B_0\} =
    \{(f'f)_{V_1\sqcup V_1'},(g'g)_{W_1\sqcup W_1'}:A_0 \rwe C_0\}.
  \end{eqnarray*}
\end{theorem}

\begin{proof}
  We follow the proof of \cite[Theorem 2.2]{murotonks08}.
  
  Let $G$ be the image of $\<\cdot,\cdot\>$ and note that $\curD_1\E_\dot/G$ is
  abelian by relation (SQ2).  We first show that every element $x$ in
  $\curD_1\E_\dot/G$ can be represented in this way.  Indeed, note from Lemma
  \ref{lem:Brels}(B4) that
  \[[A_0 \rwe B_0][C_0 \rwe D_0] \equiv_G 
  [A_0\sqcup C_0 \rwe B_0 \sqcup D_0]\]
  and from Lemma \ref{lem:Brels}(B5) 
  \[[A_1][B_1] \equiv_G 
  [B_1 \sqcup A_1].\]
  Thus we can write 
  \begin{align*}
    x &= [A_0 \rwe B_0]^{-1}[C_1]^{-1}[D_1][E_0 \rwe F_0] \\
    &\equiv_G [A_0 \sqcup F_0 \rwe B_0 \sqcup
    F_0]^{-1}[C_1 \sqcup s_0d_0D_1]^{-1}[s_0d_0C_1 \sqcup D_1][B_0 \sqcup E_0 \rwe B_0 \sqcup F_0].
  \end{align*}
  Write $X_0 = A_0 \sqcup F_0$, $Y_0 = B_0 \sqcup E_0$ and $Z_0 = B_0 \sqcup
  F_0$.  Also write $V_1 = C_1 \sqcup s_0d_0D_1$ and $W_1 = s_0d_0C_1 \sqcup
  D_1$.  Then $d_0V_1 = d_0W_1$, and we have
  \[x \equiv_G [X_0 \rwe Z_0]^{-1}[V_1]^{-1}[W_1][Y_0 \rwe Z_0].\]
  Then in $(\curD_0\E_\dot)^{ab}$ we have
  \begin{align*}
    \partial x &=
    [X_0]^{-1}[Z_0][d_1V_1]^{-1}[d_0V_1][d_0W_1]^{-1}[d_1W_1][Z_0]^{-1}[Y_0] \\
    &= [X_0]^{-1}[d_1V_1]^{-1}[d_1W_1][Y_0].
  \end{align*}
  The group $(\curD_0\E_\dot)^{ab}$ is the free abelian group on the elements of
  $S_0$ (the set of generators for the object set of $\E_0$); thus if $\partial
  x = 1$ there exists an isomorphism $X_0\sqcup d_1V_1 \rto^{\cong} Y_0 \sqcup
  d_1W_1$ which is a permutation of factors in a coproduct.  This isomorphism is
  in the image of $\<\cdot,\cdot\>$.  Thus modulo $G$,
  \[[Y_0 \rwe Z_0] \equiv_G [Y_0 \sqcup d_1W_1 \rwe Z_0 \sqcup d_1W_1][X_0
  \sqcup d_1V_1 \rwe Y_1 \sqcup d_1W_1] = [X_0 \sqcup d_1V_1 \rwe Z_0 \sqcup
  d_1W_1].\] If we write $X_0' = X_0 \sqcup d_1V_1$, $f:X_0' \rwe Z_0 \sqcup
  d_1V_1$ and $g:X_0' \rwe d_1W_1$ we can write
  \begin{align*}
    x &\equiv_G [X_0 \sqcup d_1V_1 \rwe Z_0 \sqcup
    d_1V_1]^{-1}[V_1]^{-1}[W_1][X_0 \sqcup d_1V_1 \rwe Z_0 \sqcup d_1W_1] \\
    &= \{f_{V_1},g_{W_1}:X_0' \rwe Z_0\}.
  \end{align*}
  Since $\partial\{f_{V_1},g_{W_1}:X'_0\rwe Z_0\} = 1$, we can write
  \[x = \{f_{V_1},g_{W_1}:X_0' \rwe Z_0\}\alpha,\]
  where $\partial \alpha = 1$ and $\alpha$ is in $G$.  (We can assume $\alpha$ is on
  the right because $G$ is central.)  $\curD_0\E_\dot$ is a free group of
  nilpotency class $2$, so we can apply \cite[Lemma 5.2]{murotonks08} to
  conclude that $\alpha$ must be of the form $\<y,y\>$ for some $y$ in
  $\curD_0\E_\dot$.  But $\<y,y\>$ only depends on $y$ mod $2$, so we can assume
  that $y$ is a sum of objects of $\E_0$, and thus that it is equal to $[A_0]$
  for some $A_0$ in $\E_0$.  Note, however, that 
  \[\ms{\langle [A_0],[A_0]\rangle = \{(\tau_{A_0,A_0})_0,1_0:A_0 \sqcup A_0 \rwe
    A_0 \sqcup A_0\}.}\]
  Therefore $x$ is the product of two elements of the desired form, and by
  Lemma~\ref{lem:addingcurly} it can also be represented in such a form.

  The first relation follows from the definition of $\{f_{V_1},f_{V_1}:A_0 \rwe
  B_0\}$.  The second and third follow from Lemma~\ref{lem:addingcurly}.  
\end{proof}

\begin{remark}
  Let $\E_\dot$ be a simplicial Waldhausen category in which every weak
  equivalence in $\E_0$ is a cofibration, and in which weak equivalences are
  preserved under pushouts.  Then if we have a pushout square
  \begin{diagram}
    { A_0 & B_0 \sqcup d_1V_1 \\
      B_0 \sqcup d_1W_1 & C_0 \\};
    \acyccofib{1-1}{1-2}^{f} \acyccofib{1-1}{2-1}_g 
    \acyccofib{1-2}{2-2}^{g'} \acyccofib{2-1}{2-2}^{f'}
  \end{diagram}
  then the element in $K_1(\E_\dot)$ represented by 
  \[\{f_{V_1},g_{W_1}:A_0 \rwe B_0\}\]
  is also represented by 
  \[[B_0 \sqcup d_1V_1 \rto^{f'} C_0]^{-1}[V_1]^{-1}[W_1][B_0 \sqcup d_1W_1
  \rto^{g'} C_0].\]
  Thus the objects $V_1$ and $W_1$ can be chosen to modify either the domains or
  the codomains of the pair of morphisms.
\end{remark}

The key observation for using Theorem~\ref{thm:K1diag} to work with $K_1$ of an
assembler is that for any simplicial closed assembler $\C_\dot$, $\SC(\C_\dot)$
satisfies the conditions of Proposition~\ref{prop:1typeE.}.  For this it is
important that we chose a model for $\FinSet$ that has a strictly associative
coproduct.

We can now use the Waldhausen structure on $\SC(\C)$ to get generators and
relations for $K_1(\C)$.  We start by examining a more general case first, as it
will be useful in \cite{Z-ass-var}.  Recall that $\W(\C) = (w\SC(\C))^\op$; we
use this notation here to avoid clutter.

For any morphism of assemblers $F: \D \rto \C$ there exists a simplicial
assembler $(\C/F)_\dot$, together with a moprhism $C \rto (\C/F)_\dot$ of
simplicial assemblers such that 
\[K(\D) \rto^{K(F)} K(\C) \rto K((\C/F)_\dot)\]
is a cofiber sequence.  The simplicial assembler has 
\[(\C/F)_0 = \C \qquad\hbox{and}\qquad (\C/F)_1 = \C \vee \D\] with $d_i|_\C =
s_i|\C = 1_\C$ and $d_0|_\D = F$.  For a more detailed discussion of
$(\C/F)_\dot$ and the cofiber sequence, see \cite[Section 6]{Z-Kth-ass}.

\begin{corollary} \lbl{cor:K1Cg}
  Let $\iota:\D \rto \C$ be an inclusion of a subassembler.  Then
  $K_1((\C/\iota)_\dot)$ is generated by diagrams in $\W(\C)$
    \[\mory{A}{B}{C}{D}{f}{g}\]
    for $A,B$ in $\W(\C)$ and $C,D$ in $\W(\D)$.  These satisfy the relations
    \[\bmory{A}{B}{C}{C}{f}{f} = 0 ,\]
    \[\bmory{A}{B}{C}{D}{f}{g} + \bmory{B}{B'}{C'}{D'}{f'}{g'} =
    \bmory{A}{B'}{C\sqcup C'}{D\sqcup D'}{f'f}{g'g}\]
    and
    \[\bmory{A}{B}{C}{D}{f}{g} + \bmory{A'}{B'}{C'}{D'}{f'}{g'} = \bmory{A\sqcup
      A'}{B\sqcup B'}{C\sqcup C'}{D\sqcup D'}{f\sqcup f'}{g \sqcup g'}.\]
\end{corollary}

We do not claim that these are \textsl{all} of the relations that these
generators satisfy.  As mentioned in \cite{murotonks08}, it is conjectured that
these are all of the relations, but proving this is more difficult.

\begin{proof}
  Since $\SC((\C/\iota)_1) = \SC(\C\vee\D)$, it naturally contains a copy of
  $\SC(\C)$ and one of $\SC(\D)$.  In the course of this proof we will consider
  $\SC(\C)$ and $\SC(\D)$ as subcategories of $\SC(\C\vee\D)$.
  
  We think of a diagram 
  \[\mory{A}{B}{C}{D}{f}{g}\]
  as representing the element
  \[\{f_{\SC(\iota)(C)\sqcup D}:A \rwe B\}^{-1} \{g_{C\sqcup \SC(\iota)(D)}: A \rwe
  B\}.\] Thus to check that $K_1((\C/\iota)_\dot)$ is generated by the given
  diagrams it suffices to check that every element can be represented by such a
  diagram.  Let $x$ in $K_1((\C/\iota)_\dot)$ be any element, and pick a
  representative
  \[\{f_{V_1}:A_0 \rwe B_0\}^{-1}\{g_{W_1}:A_0 \rwe B_0\}.\] 
  Write $V_1 = \{v_i\}_{i\in I}$ and $W_1 = \{w_j\}_{j\in J}$.  Then $d_0V_1 =
  \{v_i\}_{i\in I}$ and $d_0W_1 = \{\iota w_j\}_{j\in J}$.  Since $d_0V_1 =
  d_0W_1$ it follows that $J = I$ and for all $i\in I$, $v_i = \iota w_i$.  Define
  a partition of $I$ by
  \begin{align*}
    I^{\C\C} &= \{i\in I\,|\, v_i\in \C,\ w_i\in \C\}, \\
    I^{\C\D} &= \{i\in I\,|\, v_i\in \C,\ w_i\in \D\}, \\
    I^{\D\C} &= \{i\in I\,|\, v_i\in \D,\ w_i\in \C\},\ \mathrm{and} \\
    I^{\D\D} &= \{i\in I\,|\, v_i\in \D,\ w_i\in \D\}.
  \end{align*}
  For $\chi = \C\C,\C\D,\D\C,\D\D$ we define $V_1^\chi = \{v_i\}_{i\in I^\chi}$ and
  $W_1^\chi = \{w_i\}_{i\in I^\chi}$.  For conciseness we also write $V_1' =
  \{v_i\}_{i\in I \bs I^{\C\C}}$ and $W_1' = \{w_i\}_{i\in I\bs I^{\C\C}}$, so
  that $d_0V_1' = d_0W_1'$. 
  Then we have
  \[V_1^{\C\C} = W_1^{\C\C} \quad V_1^{\C\D} = \SC(\iota)(W_1^{\C\D}) \quad
  W_1^{\D\C} = \SC(\iota)(V_1^{\D\C}) \quad V_1^{\D\D} = W_1^{\D\D}.\] Note that
  $s_0d_0V_1^{\C\C} = V_1^{\C\C}$.  Write $Z = d_0V_1^{\C\C} = d_1V_1^{\C\C}$,
  and let $\varphi: I \rto I^{\C\C} \sqcup (I \bs I^{\C\C})$ be the natural
  isomorphism.  Let $\varphi_{V_1}$ and $\varphi_{W_1}$ be the induced isomorphisms
  \[\varphi_{V_1}:V_1 \rto^{\cong} s_0Z \sqcup V_1' \qqand \varphi_{W_1}:W_1 \rto^{\cong} s_0Z \sqcup W_1'.\]    It
  follows that
  \begin{eqnarray*}
    \{f_{V_1}:A_0 \rwe B_0\} 
    &=& [V_1][B_0 \sqcup Z \sqcup d_1V_1' \rto^{B_0\sqcup d_1\varphi_{V_1}^{-1}}
    B_0\sqcup d_1V_1][A_0 \rto^{(B_0\sqcup \varphi_{V_1})f_{V_1}} B_0
    \sqcup Z \sqcup d_1V_1'] \\
    &=& [V_1][Z \sqcup d_1V_1' \rto^{d_1\varphi_{V_1}^{-1}} d_1V_1][A_0 \rwe B_0 \sqcup Z \sqcup
    d_1V_1'] \\
    &=& [Z\sqcup d_0V_1' \rto^{d_0\varphi_{V_1}^{-1}} d_0V_1][s_0Z \sqcup V_1'][A_0
    \rwe B_0 \sqcup Z \sqcup d_1V_1'].
  \end{eqnarray*}
  where the last line uses relation (A5).  Thus 
  \[x = \{f_{V_1'}:A_0 \rwe B_0 \sqcup Z\}^{-1}[Z \sqcup d_0V_1'
  \rto^{d_0\varphi_{V_1}\circ d_0\varphi_{W_1}^{-1}} Z
  \sqcup d_0W_1']\{g_{W_1'}:A_0 \rwe B_0 \sqcup Z\}.\]
  Here, we used that $d_0V_1' = d_0W_1'$ to compose the two isomorphisms we get in
  the middle.  By construction, the  middle isomorphism is induced by the
  identity on the indexing set and is therefore the identity, so that
  we get
  \[x = \{f_{V_1'},g_{W_1'}:A_0 \rwe B_0 \sqcup Z\}.\] Thus we can always choose
  a representative for $x$ in which $V_1^{\C\C} = W_1^{\C\C} = 0$; from this point
  on, we assume that this was originally the case for $x$.
  
  By using a similar construction to the one for $\varphi$ above, we can assume
  that $V_1' = V_1^{\C\D} \sqcup V_1^{\D\C} \sqcup V_1^{\D\D}$ and similarly for
  $W_1'$.  Let $V_1'' = V_1^{\C\D} \sqcup V_1^{\D\C}$ and $W_1'' = W_1^{\C\D}
  \sqcup W_1^{\D\C}$, and let $Z' = V_1^{\D\D} = W_1^{\D\D}$.  Note that $d_1Z'
  = 0$.   By (A7),
  \[[V_1'] = [V_1'' \sqcup Z'] = [Z'][V_1'' \sqcup 0] = [Z'][V_1''].\]
  Then we can write
  \begin{eqnarray*}
    &&\{f_{V_1'},g_{W_1'}:A_0 \rwe B_0\} \\
    &&\qquad= [A_0 \rwe B_0
    \sqcup d_1V_1'']^{-1}[V_1'']^{-1}[Z']^{-1}[Z'][W_1''][A_0 \rwe B_0 \sqcup
    d_1W_1''] \\
    &&\qquad= \{f_{V_1''},g_{W_1''}:A_0 \rwe B_0\}.
  \end{eqnarray*}
  In particular, we could also have chosen our representative for $x$ to have
  $V_1^{\D\D} = W_1^{\D\D} = 0$.
  
  Let $V = V_1^{\D\C}$ and $W = W_1^{\C\D}$.  Then 
  \[V_1'' = V \sqcup \SC(\iota)(W) \qqand W_1'' = \SC(\iota)(V) \sqcup W.\]
  This is the information contained in the desired diagram.

  The relations follow from the statement of Theorem~\ref{thm:K1diag}.
\end{proof}

\begin{remark}
  When $\iota$ is any morphism of assemblers the proof above works to show that
  $V_1^{CC}$ and $W_1^{CC}$ can be chosen to be zero; however, it is no longer
  the case that $V_1^{DD}$ and $W_1^{DD}$ can be chosen to be zero.  In fact,
  for any pair of objects $V_1^{DD}$ and $W_1^{DD}$ such that
  $\SC(\iota)(V_1^{DD}) = \SC(\iota)(W_1^{DD})$ we get a different
  representation, and the relations in $K_1$ become significantly more
  complicated.
\end{remark}

By applying this in the case when $\D$ is the trivial assembler, we get the
following: 
\begin{corollary} \lbl{cor:K1C}
    Every element of $K_1(\C)$ can be represented by a pair of morphisms
    \[\morpair{A}{B}{f}{g}\]
    in $\W(\C)$.
    These satisfy the relations 
    \[
    \big[\morpair{A}{B}{f}{f}\big] = 0,   \qquad 
    \big[\morpair{B}{C}{g_1}{g_2}\big] + \big[\morpair{A}{B}{f_1}{f_2}\big] =
    \big[\longmorpair{3em}{A}{C}{g_1f_1}{g_2f_2}\big]  \]
    and 
    \[\big[\morpair{A}{B}{f_1}{f_2}\big] + \big[\morpair{C}{D}{g_1}{g_2}\big] =
    \big[\begin{tikzpicture}[baseline] \node (A) at (0,0) {$A\amalg C$}; \node
      (B) at (5em,0) {$B\amalg D$}; \diagArrow{->, bend left}{A}{B}^{f_1\amalg g_1} \diagArrow{->,
        bend %
        right}{A}{B}_{f_2\amalg g_2}%
    \end{tikzpicture}\big]
    \]
\end{corollary}

In addition, we can use the description in Theorem~\ref{thm:K1diag} to give a
formula for the boundary map  $K_1((\C/\iota)_\dot) \rto K_0(\D)$.  Since
$(\C/\iota)_1 = \C \vee \D$, any object $A$ in $\SC((\C/\iota)_1)$ can be write as
$A_\C \sqcup A_\D$, where $A_\C \in \SC(\C)$ and $A_\D\in \SC(\D)$.

\begin{proposition} \lbl{prop:partial}
  Let $\iota:\D \rto \C$ be any morphism  of assemblers.
  The image of the element
  \[\alpha = \{f_{V_1},g_{W_1}:A_0 \rwe B_0\}\]
  of  $K_1((\C/\iota)_\dot)$
  in $K_0(\D)$ is 
  \[[(V_1)_\D]^{-1} [(W_1)_\D].\]
  When $\iota$ is the inclusion of a subassembler and $\alpha$ is expressed in the
  notation of Corollary~\ref{cor:K1Cg}, 
  \[\partial\Big[\mory ABCDfg\Big] = [D]-[C].\]
\end{proposition}

\begin{proof}
  By \cite[Theorem C]{Z-Kth-ass} the boundary morphism $K_1((\C/\iota)_\dot)
  \rto K_0(\D)$ is induced by the morphism of simplicial closed assemblers
  $(\C/\iota)_\dot \rto S^1\smash \D$ which collapses each copy of $\C$ to the
  initial object.  Thus by Lemma~\ref{lem:Dfunc} we just need to see what such a
  morphism of assemblers does to the representative.  Since $A_0 \rwe B_0 \sqcup
  d_1V_1$ and $A_0 \rwe B_0 \sqcup d_1W_1$ both get mapped to $1_\initial$, the
  image of $\alpha$ is equal to the image of $[V_1]^{-1}[W_1]$.  The object
  $V_1$ is a tuple of objects, some from $\C$ and some from $\D$; when we
  collapse all of $\C$ to a point we just keep the ones from $\D$; similarly,
  the image of $[W_1]$ is $[(W_1)_\D]$.
\end{proof}

This will be used to compute the differentials in the spectral sequences that
converge to $K_0[\V_k]$ and $K_0[\V_k]/([\mathbb{A}^1])$ in \cite{Z-ass-var}.

\section{Application: Interval Exchange Transformations}
\lbl{sec:iet}

Consider the assembler $\C$ whose objects are half-open intervals $[a,b)$ of the
real line.  A morphism $[a,b)\rto [c,d)$ is a real number $x$ such that
$[a+x,b+x) \subseteq [c,d)$; the topology is the usual topology.  We show that
$K_1$ of this assembler is the abelianization of the group $G$ of interval
exchange transformations (for more on inteval exchange transformations, see for
example \cite{veech84i, veech84ii, veech84iii}).  By \cite[Theorem
1.3]{veech84iii} $G^{ab}$ is isomorphic to $\R\smash_\Q \R$, the exterior tensor
square of the reals over the rationals.

First, we construct a homomorphism from $G$ to $K_1$.  An interval exchange
transformation of $[0,1)$ is a sequence of real numbers $1=a_0 < a_1 < \cdots <
a_n =1$ and a sequence of real numbers $x_i$ for $i = 1,\ldots,n$ such that
\[[x_i+a_{i-1},x_i+a_{i}) \cap [x_j + a_{j-1}, x_j + a_j) = \eset \qquad i \neq
j\] and $\bigcup_{i=1}^n [x_i+a_{i-1},x_i+a_{i}) = [0,1)$.  These data give us
an element of $K_1(\C)$
\[\big[\morpair{\{[a_{i-1},a_i)\}_{i=1}^n}{\{[0,1)\}}{\cup}{f}\big]\]
where $f_i:[a_{i-1},a_i) \rto [0,1)$ is given by the real number $x_i$ and
$\cup_i$ is given by $0$.  We need to check that the composition of elements of
$K_1(\C)$ corresponds to composition of interval exchange transformations.
Suppose that we are given two interval exchange transformations
\[\big[\morpair{A}{\{[0,1)\}}{\cup}{f}\big] \qqand
\big[\morpair{B}{\{[0,1)\}}{\cup}{g}\big].\]
Consider the diagram
\begin{diagram-fixed}
  { C' & & A \\
    C & B & \{[0,1)\}\\ 
    A & \{[0,1)\} \\};
  \to{1-1}{1-3}^{\cup} \to{1-1}{2-1}_{\cup} \to{1-3}{2-3}^\cup \to{2-1}{2-2}^{\cup} 
  \to{2-2}{3-2}^g \to{2-1}{3-1}_{g'} \to{3-1}{3-2}^\cup
  \to{2-2}{2-3}^\cup
\end{diagram-fixed}
where both squares are pullback squares, and $\cup$'s denote morphisms which are
unions of intervals.  Then
\begin{eqnarray*}
  \big[\morpair{B}{\{[0,1)\}}{\cup}{g}\big] &=&
  \big[\morpair{C'}{B}{\cup}{\cup}\big] +
  \big[\morpair{B}{\{[0,1)\}}{\cup}{g}\big] \\
  &=&     \big[\morpair{C'}{\{[0,1)\}}{\cup}{\cup g}\big] 
  = \big[\morpair{C'}{\{[0,1)\}}{\cup}{\cup g'\cup}\big] \\
  &=& \big[\morpair{C'}{A}{\cup}{\cup g'}\big] +
  \big[\morpair{A}{\{[0,1)\}}{\cup}{\cup}\big] =
  \big[\morpair{C'}{A}{\cup}{\cup g'}\big].
\end{eqnarray*}
Thus
\[\big[\morpair{A}{\{[0,1)\}}{\cup}{f}\big] +
\big[\morpair{B}{\{[0,1)\}}{\cup}{g}\big] =
\big[\morpair{C'}{\{[0,1)\}}{\cup}{\cup g'f}.\] This is the formula for the
composition of interval exchange transformations.

We now have a homomorphism $G \rto K_1(\C)$, which gives a homomorphism
$\psi:G^{ab} \rto K_1(\C)$; we show that this is an isomorphism.  First, we show
that it is surjective: every element of $K_1(\C)$ can be represented by an
interval exchange transformation.  Write $|[a,b)| = b-a$ and $[n) = [0,n)$.  For
any $\alpha\in \R$, define
\[\tau_\alpha =
\big[\morpair{\{[\alpha),[\alpha)\}}{\{[2\alpha)\}}{\tau}{1}\big],\] where
$\tau$ is given by the sequence $(\alpha,0)$ and $1$ is given by the sequence
$(0,\alpha)$.  But $2\tau_\alpha = 0$ and $\tau_\alpha = 6\tau_{\alpha/2}$, so
$\tau_\alpha = 0$ for all $\alpha$.

Any object $\SCob{A}{i}$ has a morphism to $\left\{\left[\sum_{i\in I}
    |A_i|\right)\right\}$, so
\begin{eqnarray*}
  \big[\morpair{A}{\SCob{B}{j}}{f}{g}\big] &=& \big[\morpair{A}{\SCob{B}{j}}{f}{g}\big]
  + \big[\morpair{\SCob{B}{j}}{\left[\sum_{j\in J}|B_j|\right)}{\alpha}{\alpha}\big] \\  &=&
  \big[\morpair{A}{\left[\sum_{j\in J}|B_j|\right)}{\alpha f}{\alpha g}\big]
\end{eqnarray*}
for any choice of morphism $\alpha: \SCob{B}{j} \rto \left\{\left[\sum_{j\in J}
    |B_j|\right)\right\}$.  Thus we can assume that $B$ always consists of a
single segment.

By adding in a segment to both $A$ and $B$ we can also ensure that $\sum_{j\in
  J} |B_j| = 2^n$ for some integer $n$.  In addition, since $\tau_\alpha = 0$
for all $\alpha$ we can write
\[\big[\morpair{A}{\{[2^n)\}}{f}{g}\big] =
\big[\morpair{A'}{\{[2^{n-1})\}}{f'}{g'}\big] +
\big[\morpair{A''}{\{[2^{n-1})\}}{f''}{g''}\big]\] by subdividing $A$ further
and then ensuring that each segment maps to either the first half of $[2^n)$ or
the second under both $f$ and $g$; thus $K_1(\C)$ is generated by elements of
the form
\[\morpair{A}{[1)}{f}{g}.\]
  
Write $A = \{[a_i,b_i)\}_{i\in I}$, and let $f$ be defined by the real numbers
$x_i$ and $g$ by the real numbers $y_i$.  Let $h:\{[a_i+x_i,b_i+x_i)\}_{i\in I}
\rto A$ be defined by the identity map on $I$ and the real numbers $-x_i$; then
\[\big[\morpair{A}{\{[1)\}}{f}{g}\big] +
\big[\morpair{\{[a_i+x_i,b_i+x_i)\}_{i\in I}}{A}{h}{h}\big] =
\big[\morpair{\{[a_i+x_i,b_i+x_i)\}_{i\in I}}{\{[1)\}}{hf}{hg}\big],\] where $hf
= \cup$.  Thus $\psi$ is surjective.

We now construct a surjective homomorphism $\varphi: K_1(\C) \rto \R\smash_\Q\R$
such that $\varphi\psi$ is an isomorphism, completing our proof.  To do this, we
first need to show how to construct morphisms out of $K_1(\C)$.  
\begin{lemma} \label{lem:K1mor}
  Given an abelian group $A$, there is a stable quadratic module $A_*$ with $A_0
  = 0$ and $A_1 = A$.  For any other stable quadratic module $C_*$, any homomorphism
  $f:C_1 \rto A$ such that the composition
  \[C_0^{ab}\otimes C_0^{ab} \rto^{\langle \cdot,\cdot \rangle} C_1 \rto^f A\]
  is uniformly $0$ gives rise to a morphism of stable quadratic modules $C_*
  \rto A_*$.
\end{lemma}
We omit the proof of this as it follows directly from the definitions.  From
this lemma we see that to produce a morphism $K_1(\C) \rto \R\smash_\Q\R$ it
suffices to construct a homomorphism $\curD_1\SC(\C) \rto \R\smash_\Q\R$
satisfying the condition of the lemma.
Here, we are considering $\SC(\C)$ to be a constant simplicial Waldhausen
category; in this context, the presentation of $\curD_1\SC(\C)$ is significantly
simplified, since by (A1) it is generated simply by morphisms $[A \rwe^f B]$ in
$\SC(\C)$.  We define the map $\hat\varphi: \curD_1\SC(\C) \rto \R\smash_\Q\R$
in the following manner.  Write $A = \{[a_i,a_i+\epsilon_i)\}_{i\in I}$, $B =
\{[b_j,b_j+\delta_j)\}_{j\in J}$, and suppose that $f_i:[a_i,a_i+\epsilon_i)
\rto [b_{f(i)}, b_{f(i)} + \delta_{f(i)})$ is given by translation by $x_i$.  We
define
\[\hat\varphi[A \rwe B] = \sum_{i\in I} \epsilon_i\smash x_i.\]
Note that $\hat\varphi[\tau_{A,B}] = 0$, so the relation from Lemma~\ref{lem:K1mor}
holds.  We must check that $\hat \varphi$ is a well-defined homomorphism out of
$\curD_1\SC(\C)$.  To check that it is compatible with relations (A1)-(A7)
follows directly from the definitions; the only complication comes in relation
(A3), where we must use that $\sum_{i\in f^{-1}(j)} \epsilon_i  =\delta_j$.
To check that it does not contradict a relation induced by the fact that we are
working with a stable quadratic module we note that this homomorphism is $0$ on
the image of $\langle \cdot,\cdot \rangle$; since all such relations are induced
by the interaction of $\partial$ and $\langle \cdot,\cdot \rangle$ with
commutators (which must also all map to $0$, since $\R\smash_\Q\R$ is abelian)
this homomorphism is well-defined.

We define $\varphi:K_1(\C) \rto \R\smash_\Q\R$ to be
$\hat\varphi|_{\ker\partial}$.  It remains to check that $\varphi$ is
surjective, and that the composition $\varphi\psi$ is an isomorphism.  Consider
the
element \[\big[\morpair{\{[0,\epsilon),[\epsilon,\delta)\}}{\{[0,\delta)\}}{\cup}{t}\big],\]
where $t$ is defined by the translations $(\epsilon,-\epsilon)$.  By definition,
$\varphi$ maps this to
\[-\epsilon\smash\epsilon - (\delta-\epsilon)\smash (-\epsilon) = \delta\smash
\epsilon.\] Since $\delta$ and $\epsilon$ were arbitrary, $\varphi$ is
surjective.  The homomorphism $\varphi\psi$ is exactly equal to half of the
scissors congruence invariant $S$ introduced in \cite[Equation 1.2]{veech84iii}.
Since $S$ is an isomorphism $G^{ab} \rto \R\smash_\Q\R$, $\varphi\psi$ must also
be an isomorphism.

\section{Application: Differentials in a spectral sequence}
\lbl{sec:polytope}

This section is a sequel to \cite[Section 2.2]{Z-Kth-ass}, and we liberally use
the definitions and notation from that section.  Recall that $\G_n$ is the
assembler whose objects are finite unions of open $n$-simplices in $\R^\infty$
with covering families $\{P_i \rto P\}_{i\in I}$ where $P \bs \bigcup_{i\in I}
\varphi_i(P_i)$ has dimension less than $n$.  The assembler $\gG$ has objects
finite unions of open simplices (not necessarily of the same dimension) in
$\R^\infty$; a family $\{\varphi_i:P_i \rto P\}_{i\in I}$ is a covering family
if $\bigcup_{i\in I} \varphi_i(P_i) = P$.  We write $\gG^{(n)}$ for the
subassembler of polytopes of dimension at most $n$.  Then, by \cite[Proposition
2.4]{Z-Kth-ass}
\[K(\gG^{(n-1)}) \rto K(\gG^{(n)}) \rto K(\G_n)\]
is a cofiber sequence for all $n \geq 0$.  This gives us a spectral
sequence
\[E^1_{p,q} = K_p(\G_q) \Rto K_p(\G).\]
(The indexing is inspired by the Adams spectral sequence, so that the associated
graded of each homotopy group appears in a column.)  Here $d_r:E^r_{p,q} \rto
E^r_{p-1,q-r}$.  

In particular, the $0$-th column converges to the associated graded of the
filtration on $K_0(\gG)$ induced by the images of the homomorphisms
\[\makeshort{K_0(\gG^{(0)}) \rto K_0(\gG^{(1)}) \rto \cdots \rto K_0(\gG).}\]
Thus we see that the $n$-th filtered part of $K_0(\gG)$ is a quotient of
$K_0(\gG^{(n)})$; the kernel of the quotient homomorphism $K_0(\gG^{(n)}) \rto
K_0(\gG)$ is determined by the images of the differentials $d_r:E^r_{1,m+r} \rto
E^r_{1,m}$ for all $r\geq 1$ and $m \leq n$.  In particular, $K_0(\gG)^{(n)} =
K_0(\gG^{(n)})$ if and only if all of these differentials are zero.

The differential $d_r:E^r_{1,m+r} \rto E^r_{1,m}$ is defined to be 
\[K_1(\G_{m+r}) \rto^\partial K_0(\gG^{(m+r-1)}) \rto^{\iota^{-1}}
K_0(\gG^{(m)}) \rto K_0(\G_m).\] It is only defined on those $x$ in
$K_1(\G_{m+r})$ such that $\iota^{-1}(x)$ is nonempty, and it is well-defined by the
general theory of spectral sequences.  Here, $\iota$ is the inclusion
$K_0(\gG^{(m)}) \rto K_0(\gG^{(m+r-1)})$ induced by the inclusion of assemblers
$\gG^{(m)} \rto \gG^{(m+r-1)}$.  We wish to compute this differential.

Let $x$ be in $K_1(\G_{m+r})$; $x$ must correspond to an element in
$K_1(\gG^{(m+r)}/\iota')$, where $\iota'$ is the inclusion of assemblers
$\gG^{(m+r-1)} \rto \gG^{(m+r)}$.  By Corollary~\ref{cor:K1C} $x$ is represented
by a diagram of the form
\[\morpair{\{P_i\}_{i\in I}}{\{Q\}}\cup \varphi;\]
we can assume that the domain is a single polytope and the first map is union by
the same logic employed in Section~\ref{sec:iet}.  Let 
\[T = Q \bs \bigcup_{i\in I} P_i \qqand T' = Q \bs \bigcup_{i\in I}
\varphi_i(P_i).\]
Then $T,T'\in \gG^{(m+r-1)}$ and we have an element 
\[\begin{tikzpicture}[baseline=(A.base)] \node[anchor=east] (A) at (0,0) {$\{Q\}$};
  \node[anchor=west] (B) at (2em,2ex) {$\{P_i\}_{i\in I}\sqcup \W(i)(\{T\})$}; \node[anchor=west]
  (C) at (2em,-2ex) {$\{P_i\}_{i\in I}\sqcup\W(i)(\{T'\})$}; \diagArrow{<-, bend
    left}{A}{B.west}!{\cup} \diagArrow{<-, bend right}{A}{C.west}!{f};%
  \end{tikzpicture}\]
in $K_1(\gG^{(m+r)}/\iota')$.  By Proposition~\ref{prop:partial} 
\[\partial x = [T'] - [T].\]
Since $\iota^{-1}(x) \neq \eset$ this means that we can write $[T']-[T] =
[R']-[R]$ with $\dim R',\dim R \leq m$.  The projection to $K_0(\G_m)$ takes
this to the difference $[\mathring{R'}] - [\mathring{R}]$, where $\mathring{R}$
denotes the $m$-dimensional interior.

For example, when $m=0$ and $r=1$ this takes an interval exchange transformation
which splits a segment $[a,b]$ into $n$ segments, thought of as an
almost-everywhere defined injective piecewise isometry $[a,b] \rto [a,b]$ and
sends it to the difference between the number of points where it is undefined in
the domain and the codomain.  Since these numbers must be the same, we see that
$d_1:E^1_{1,1} \rto E^1_{0,0}$ is uniformly zero.  It is unknown whether there
are nonzero differentials for other values of $m$ and $r$.

\section{Generalizing the construction of Muro and Tonks}
\lbl{app:murotonks}

The goal of this section is to prove a generalization of the construction given
by Muro and Tonks in \cite{murotonks07,murotonks08} of the $1$-type of a
Waldhausen category and prove Proposition~\ref{prop:1typeE.}.  To get their
construction, Muro and Tonks take the bisimplicial set $X_{\dot\dot} = N_\dot
wS_\dot\E$ for a Waldhausen category $\E$, and then compute a representation of
its fundamental crossed complex $\pi X_{\dot\dot}$.  They then note that this
crossed complex inherits a monoid structure from the symmetric monoidal
structure (induced by coproduct) on $\E$, and use it to construct a stable
quadratic module determined by $X_{\dot\dot}$; $K_1(\E)$ is then the kernel of
the boundary map of the stable quadratic module.

To generalize this to simplicial Waldhausen categories, we note that the key
features of a Waldhausen category used by Muro and Tonks are the following:
\begin{itemize}
\item When using the $S_\dot$ construction, $K(\E)$ is an $\Omega$-spectrum
  above level $1$, so identifying the $2$-type of $K(\E)_1$ is equivalent to
  identifying the $1$-type of $K(\E)$.
\item $X_{\dot\dot}$ is horizontally reduced, in the sense that $X_{0\dot} =
  \Delta^0$.
\item $X_{\dot\dot}$ has a strictly unital monoid structure that comes from
  the coproduct structure on $\E$.
\end{itemize}

Let $\E_\dot$ be a simplicial Waldhausen category.  We define a bisimplicial set
$Y_{\dot\dot}$ by
\[Y_{mn} = N_n wS_m \E_n.\]
Then $Y_{\dot\dot}$ has all of the same properties that $X_{\dot\dot}$ has,
above, and thus the $1$-type of $K(\E_\dot)$ can be recovered from the
fundamental crossed complex of $K(\E_\dot)_1$.

Using the generators and relations given in \cite[Lemma~4.6]{murotonks07}
together with the methods described on \cite[page~18]{murotonks07} we can define
the stable quadratic module $D_*\E_\dot$  associated to a simplicial Waldhausen
category in the following manner.  The module $D_*\E_\dot$ is closely related
to, but not isomorphic to, $\curD_*\E_\dot$; we use similar notation to
emphasize this fact.

\begin{definition}
  Let $\E_\dot$ be a simplicial Waldhausen category.  The stable quadratic
  module $D_*\E_\dot$ is defined as follows.  $D_0\E_\dot$ has generators
  $[A_0]$ for $A_0$ in $\ob\E_0$.  $D_1\E_\dot$ has generators
  \begin{itemize}
  \item  $[A_1 \rwe B_1]$ for weak equivalences $A_1\rwe B_1$ in $w\E_1$, and
  \item $[A_0 \rcofib B_0 \rfib B_0/A_0]$ for cofiber sequences $A_0 \rcofib B_0
    \rfib B_0/A_0$ in $\E_0$.
  \end{itemize}
  We use subscripts to keep track of the simplicial dimension of a generator.
  Thus an object $A_i$ lives in $\E_i$.  These satisfy the following relations:
  \begin{itemize}
  \item[(R1)] $\partial[A_1 \rwe B_1] = [d_0B_1]^{-1}[d_1A_1]$.
  \item[(R2)] $\partial[A_0 \rcofib B_0 \rfib B_0/A_0] =
    [B_0]^{-1}[B_0/A_0][A_0]$.
  \item[(R3)] $[0] = 1$.
  \item[(R4)] $[s_0A_0 \req s_0A_0] = 1$.
  \item[(R5)] $[A_0 \req A_0 \rfib 0] = [0 \rcofib A_0 \req A_0] = 1$.
  \item[(R6)] For any pair of composable weak equivalences $A_2 \rwe B_2 \rwe
    C_2$ in $\E_2$, 
    \[[d_1A_2 \rwe d_1C_2] = [d_0B_2 \rwe d_0C_2][d_2A_2 \rwe d_2B_2].\]
  \item[(R7)] For any commutative diagram
    \begin{diagram}
      { A_1 & B_1 & B_1/A_1 \\ A_1' & B_1' & B_1'/A_1' \\};
      \cofib{1-1}{1-2} \cofib{2-1}{2-2} 
      \fib{1-2}{1-3} \fib{2-2}{2-3}
      \we{1-1}{2-1} \we{1-2}{2-2} \we{1-3}{2-3}
    \end{diagram}
    in $\E_1$, the element
    \[[A_1 \rwe A_1'][B_1/A_1 \rwe B_1'/A_1']^{[d_1A_1]}\]
    is equal to 
    \[[d_0A_1' \rcofib
    d_0B_1' \rfib d_0(B_1'/A_1')]^{-1}[B_1 \rwe B_1'][d_1A_1 \rcofib d_1B_1
    \rfib d_1(B_1/A_1)].\]
  \item[(R8)] For any commutative diagram
    \begin{diagram}
      { & & C_0/B_0 \\ & B_0/A_0 & C_0/A_0 \\ A_0 & B_0 & C_0 \\};
      \cofib{3-1}{3-2} \cofib{3-2}{3-3} \cofib{2-2}{2-3}
      \fib{3-2}{2-2} \fib{3-3}{2-3} \fib{2-3}{1-3}
    \end{diagram}
    the element
    \[[A_0 \rcofib C_0 \rfib C_0/A_0][B_0/A_0 \rcofib C_0/A_0 \rfib
    C_0/B_0]^{[A_0]}\]
    is equal to 
    \[[B_0 \rcofib C_0 \rfib C_0/B_0][A_0 \rcofib B_0 \rfib
    B_0/A_0].\]
  \item[(R9)] For any pair of objects $A_0,B_0$ in $\E_0$, 
    \[\langle [A_0],[B_0] \rangle = [B_0 \rcofib A_0\sqcup B_0 \rfib A_0]^{-1}[A_0 \rcofib
    A_0 \sqcup B_0 \rfib B_0].\]  
  \end{itemize}
  (We use analogous names and relations to \cite{murotonks07,murotonks08}.
  However, we use multiplicative and not additive notation to emphasize the fact
  that these groups are not abelian.)

  The stable quadratic module $D^+_*\E_\dot$ is the quotient of $D_*\E_\dot$ by
  the additional relation
  \begin{itemize}
  \item[(R10)] $[B \rcofib A\sqcup B \rfib A] = 0$.
  \end{itemize}
  In $D^+_*\E_\dot$, (R9) is equivalent to the relation $\<[A_0],[B_0]\> =
  [s_0B_0 \sqcup s_0A_0 \rto^{\cong} s_0A_0 \sqcup s_0B_0]$.  In addition, (R10)
  implies that $[A_0 \sqcup B_0] = [A_0][B_0]$.
\end{definition}

Then the proof for \cite[Theorem 1.7, Corollary 1.9]{murotonks07} works to show
that the stable quadratic module obtained from $Y_{\dot\dot}$ gives the $1$-type
of $K(\E_\dot)$; it turns out that this stable quadratic module is $D_*\E_\dot$.
If $\E_0$, $E_1$ and $\E_2$ have a strictly associative and unital coproduct
which is compatible with the simplicial structure maps we can apply
\cite[Theorem 4.2]{murotonks08} (whose proof works identically for a simplicial
Waldhausen category) to get the following:

\begin{theorem} \lbl{thm:genmurotonks}
  For any simplicial Waldhausen category $\E_\dot$ such that $\E_0$, $\E_1$ and
  $\E_2$ have a strictly associative and unital coproduct compatible with the
  simplicial structure maps, the stable quadratic module $D^+_*\E_\dot$
  satisfies
  \[K_1(\E_\dot) \cong \ker \partial \qquad\hbox{and}\qquad K_0(\E_\dot) \cong
  \coker \partial,\] where $\partial$ is the boundary map in $D^+_*\E_\dot$.  In
  addition, the first Postnikov invariant can be obtained fro the structure of
  $D^+_*\E_\dot$.  
\end{theorem}

\begin{lemma} \lbl{lem:Brels}
  In $\curD_*\E_\dot$ the following extra relations hold:
  \begin{itemize}
  \item[(B1)] $[\tau_{A_0,B_0}]$ is central for all $A_0,B_0$ in $\E_0$.
  \item[(B2)] $[\tau_{A_0,B_0\sqcup C_0}] = [\tau_{A_0,B_0}][\tau_{A_0,C_0}]$
    and $[\tau_{A_0\sqcup B_0,C_0}] = [\tau_{A_0,C_0}][\tau_{B_0,C_0}]$.
  \item[(B3)] For any weak equivalence $A_0 \rwe B_0$ in $\E_0$, any object
    $C_0$ in $\E_0$ and any object $A_1$ in $\E_1$ 
    \[[A_0 \rwe B_0]^{[C_0]} = [A_0 \sqcup C_0 \rwe B_0 \sqcup C_0]
    \qquad\hbox{and}\qquad [A_1]^{[C_0]} = [A_1 \sqcup s_0C_0].\]
  \item[(B4)] For any two weak equivalences $A_0 \rwe B_0$ and $C_0 \rwe D_0$ in
    $\E_0$, 
    \[\ms{[A_0 \rwe B_0][C_0\rwe D_0] = [A_0 \sqcup C_0 \rwe B_0 \sqcup D_0]}[\tau_{D_0,B_0}][\tau_{A_0,D_0}].\]
  \item[(B5)] For any two objects $A_1,B_1$ in $\E_1$,
    \[[A_1][B_1] = [A_1\sqcup
    B_1][\tau_{d_0B_1,d_0A_1}][\tau_{d_1A_1,d_0B_1}].\]
  \item[(B6)] For any three objects $A_0,B_0,C_0$ in $\E_0$,
    \[[\tau_{A_0,B_0} \sqcup C_0] = [\tau_{A_0,B_0}].\]
  \end{itemize}
\end{lemma}

\begin{proof}
  Relations (B1) and (B2) follow directly from the definition of a stable
  quadratic module.  The others follow straightforwardly from relations
  (A1)-(A7) as well as (B1) and (B2).  As an example, we prove (B4); the others
  follow similarly.
    \[
    \begin{array}{l}
      [A_0 \rwe B_0][C_0 \rwe D_0] \\
      \qquad\anno{(A6)}= \big([\tau_{D_0,B_0}][\tau_{B_0,D_0}]\big)[D_0 \sqcup A_0 \rwe D_0 \rwe
      B_0]\big([\tau_{D_0,A_0}][\tau_{A_0,D_0}]\big)[A_0 \sqcup C_0 \rwe A_0 \sqcup D_0]
      \\
      \qquad\anno{(B1)}= \big([D_0 \sqcup B_0 \rwe B_0 \sqcup D_0][D_0 \sqcup A_0 \rwe D_0
      \sqcup B_0][A_0 \sqcup D_0 \rwe D_0 \sqcup A_0]\big)\\\qquad\qquad[A_0 \sqcup C_0 \rwe
      A_0 \sqcup D_0][\tau_{D_0,B_0}][\tau_{A_0,D_0}] \\
      \qquad\anno{(A3)}= [A_0 \sqcup C_0 \rwe B_0 \sqcup D_0][\tau_{D_0,B_0}][\tau_{A_0,D_0}].
    \end{array}\]
\end{proof}

We are now ready to prove Proposition~\ref{prop:1typeE.}.

\begin{proof}[Proof of Proposition~\ref{prop:1typeE.}]
  By Theorem~\ref{thm:genmurotonks} it suffices to show that $D_*^+\E_\dot\cong
  \curD_*\E_\dot$.  We define the homomorphism $f:D_*^+\E_\dot \rto
  \curD_*\E_\dot$ by
  \begin{align*}
    f([A_0]) &= [A_0], \\
    f([A_1 \rwe B_1]) &= [B_1][d_1A_1 \rwe d_1B_1], \hbox{ and} \\
    f([A_0 \rcofib B_0 \rfib B_0/A_0]) &= [B_0/A_0 \sqcup A_0 \rwe B_0].
  \end{align*}
  Let $g:\curD_*\E_\dot \rto D_*^+\E_\dot$ be defined by
  \begin{align*}
    g([A_0]) &= [A_0], \\
    g([A_1]) &= [A_1 \req A_1], \hbox{ and} \\
    g([A_0 \rwe B_0]) &= [s_0A_0 \rwe s_0B_0].
  \end{align*}
  On the generators in degree $0$ $f$ and $g$ are clearly inverses of one
  another.  Note that the following equalities hold.
  \begin{align*}
    fg([A_1]) &\anno{\phantom{(A1)}}{=} f([A_1 \req A_1]) = [A_1][d_1A_1 \req d_1A_1] \anno{(A1)}{=} [A_1]. \\
    fg([A_0 \rwe B_0]) &\anno{\phantom{(A1)}}{=} f([s_0A_0 \rwe s_0B_0]) = [s_0B_0][d_1s_0A_0 \rwe
    d_1s_0B_0] \\
    &\anno{(A1)}{=} [A_0 \rwe B_0]. \\
    gf([A_1 \rwe B_1]) & \anno{\phantom{(A1)}}{=} g([B_1][d_1A_1 \rwe d_1B_1]) =
    [B_1 \req B_1][s_0d_1A_1 \rwe s_0d_1B_1]. \\
    gf([A_0 \rcofib B_0 \rfib B_0/A_0]) & \anno{\phantom{(A1)}}{=} g([B_0/A_0
    \sqcup A_0 \rwe B_0]) = [s_0B_0/A_0 \sqcup s_0A_0 \rwe s_0B_0].
  \end{align*}
  Let $A_1 \rwe B_1$ be a weak equivalence in $\E_1$.  Then in $\E_2$ we
  have the composite 
  \[s_0A_1 \rwe s_0B_1 \req s_0B_1;\]
  applying (R6) we get that
  \[[A_1 \rwe B_1] = [B_1 \req B_1][d_2s_0A_1 \rwe d_2s_0B_1] = [B_1 \req
  B_1][s_0d_1A_1 \rwe s_0d_1B_1].\]
  Thus $gf([A_1 \rwe B_1]) = [A_1\rwe B_1]$.  Now let $A_0 \rcofib B_0 \rfib
  B_0/A_0$ be a cofiber sequence in $\E_0$.  Then we have the following diagram
  in $\E_1$:
  \begin{diagram}
    {s_0A_0 & s_0B_0/A_0 \sqcup s_0A_0 & s_0B_0/A_0 \\
      s_0A_0 & s_0B_0 & s_0B_0/A_0 \\};
    \cofib{1-1}{1-2} \cofib{2-1}{2-2} \fib{1-2}{1-3} \fib{2-2}{2-3}
    \eq{1-1}{2-1} \we{1-2}{2-2} \eq{1-3}{2-3}
  \end{diagram}
  Applying (R7) and simplifying using (R4) and (R10) we get
  \[1 = [A_0 \rcofib B_0 \rfib B_0/A_0]^{-1}[s_0B_0/A_0 \sqcup A_0 \rwe
  s_0B_0].\] Thus $gf([A_0 \rcofib B_0 \rfib B_0/A_0]) = [A_0 \rcofib B_0 \rfib
  B_0/A_0]$.  Thus in order to show that $f$ and $g$ are inverse isomorphisms it
  remains to check that they are well-defined. 

  The function $g$ clearly has no choices to be made in its definition, but the
  definition of $f$ has the possibility of a choice in how to split weak
  equivalences.  However, the compatibility condition ensures that this is not
  the case.  First, note that for any cofiber sequence of the form $A \req A
  \rfib 0$ $\alpha$ must be the identity on $A$ (since the disjoint union is
  strictly unital).  Now suppose that for some cofiber sequence $A \rcofib B
  \rfib B/A$ there exist two choices $\alpha$ and $\alpha'$ that both satisfy
  condition (3).  We consider the diagram
  \begin{diagram}
    {& & 0\\
      & B/A & B/A \\
      A & B & B\\};
    \cofib{3-1}{3-2}^f \eq{3-2}{3-3} \cofib{2-2}{2-3}
    \fib{3-2}{2-2} \fib{3-3}{2-3} \fib{2-3}{1-3}
  \end{diagram}
  where we think of the splitting for $f$ as being $\alpha$ and the splitting of
  $1_B \circ f$ as being $\alpha'$.  Then the compatibility condition ensures
  that $\alpha = \alpha'$, so $f$ is well-defined on elements.  Checking that
  $f$ and $g$ commute with $\partial$ and $\<\cdot,\cdot\>$ is straightforward,
  and thus they both also commute with the action of the $0$-level on the $1$-level.

  It remains to check that all relations in $D_*^+\E_\dot$ and $\curD_*\E_\dot$
  are preserved by $f$ and $g$, respectively.  We check that applying $f$ to
  each of the relations (R3)-(R8) and (R10) gives a valid equality in
  $\curD_*\E_\dot$.  Relations (R3)-(R5) follow directly from the definitions.
  The other relations follow directly from relations (A1)-(A7), with judicious
  use of (B3) whenever the action of $D_0^+\E_\dot$ on $D_1^+\E_\dot$ is needed.
  As an example of this kind of computation, we prove (R8); the others follow
  analogously.  The key idea in all of the proofs is to compose as many
  morphisms as possible and use (A5) to commute the object-type generators past
  the morphism-type generators.  

  Suppose that we have a commutative diagram
  \begin{diagram}
    { & & C_0/B_0 \\ & B_0/A_0 & C_0/A_0 \\ A_0 & B_0 & C_0 \\};
    \cofib{3-1}{3-2} \cofib{3-2}{3-3} \cofib{2-2}{2-3}
    \fib{3-2}{2-2} \fib{3-3}{2-3} \fib{2-3}{1-3}
  \end{diagram}
  in $\E_0$.  We have
  \[\begin{array}{l}
    f([A_0 \rcofib C_0 \rfib C_0/A_0][B_0/A_0 \rcofib C_0/A_0 \rfib
    C_0/B_0]^{[A_0]}) \\
    \qquad= [C_0/A_0 \sqcup A_0 \rwe C_0][C_0/B_0 \sqcup B_0/A_0 \rwe
    C_0/A_0]^{[A_0]} \\
    \qquad\anno{(B3)}= [C_0/A_0 \sqcup A_0 \rwe C_0][C_0/B_0\sqcup B_0/A_0 \sqcup A_0 \rwe C_0/A_0
    \sqcup A_0] \\
    \qquad= [C_0/B_0 \sqcup B_0/A_0 \sqcup A_0 \rwe C_0].
  \end{array}\]
  On the other hand, we have
  \[\begin{array}{l}
    f([B_0\rcofib C_0 \rfib C_0/B_0][A_0\rcofib B_0 \rfib B_0/A_0]) = [C_0/B_0
    \sqcup B_0 \rwe C_0][B_0/A_0 \sqcup A_0 \rwe B_0] \\
    \qquad \anno{(A6)}=[C_0/B_0\sqcup B_0 \rwe C_0][C_0/B_0 \sqcup B_0/A_0
    \sqcup A_0 \rwe C_0/B_0 \sqcup B_0] \\
    \qquad \anno{(A3)}= [C_0/B_0 \sqcup B_0/A_0 \sqcup A_0 \rwe C_0],
  \end{array}\]
  and this must be the same morphism as in the previous part by the
  compatibility condition (3).

  To check $g$ we need to show that relations (A1)-(A7) still hold after
  applying $g$.  We summarize the main steps of checking each relation.
  \begin{itemize}
  \item[(A1)] Follows directly from (R4).
  \item[(A2)] The key step here is that $[A_0\sqcup B_0] = [A_0][B_0]$ by
    applying $\partial$ to (R10).
  \item[(A3)] For any two composable weak equivalences $A_0 \rwe B_0 \rwe C_0$
    in $\E_0$, apply (R6) to the composition $s_0s_0 A_0 \rwe s_0s_0B_0 \rwe
    s_0s_0 C_0$.
  \item[(A4)] Applying (R6) to the composition $A_2 \req A_2 \req A_2$ gives the
    relation $[d_1A_2 \req d_1A_2] = [d_0A_2 \req d_0A_2][d_2A_2 \req d_2A_2]$.
  \item[(A5)] For any weak equivalence $A_1 \rwe B_1$ in $\E_1$, apply (R6) to
    the two compositions $s_0A_1 \rwe s_0B_1 \req s_0B_1$ and $s_1A_1 \req
    s_1A_1 \rwe s_1B_1$.
  \item[(A6)] For any weak equivalence $A_1 \rwe B_1$ in $\E_1$ and any object
    $C_0$ in $\E_0$ apply (R7) to the diagram
    \begin{diagram}
      { A_1 & s_0C_0 \sqcup A_1 & s_0C_0 \\ 
        B_1 & s_0C_0 \sqcup B_1 & s_0C_0
        \\};
      \cofib{1-1}{1-2} \cofib{2-1}{2-2} \fib{1-2}{1-3} \fib{2-2}{2-3}
      \we{1-1}{2-1} \we{1-2}{2-2} \eq{1-3}{2-3}
    \end{diagram}
    and simplify using (R4) and (R10).
  \item[(A7)] To the diagram
    \begin{diagram}
      { B_1 & A_1 \sqcup B_1 & A_1 \\ 
        B_1 & A_1 \sqcup B_1 & A_1
        \\};
      \cofib{1-1}{1-2} \cofib{2-1}{2-2} \fib{1-2}{1-3} \fib{2-2}{2-3}
      \eq{1-1}{2-1} \eq{1-2}{2-2} \eq{1-3}{2-3}
    \end{diagram}
    apply (R7) and and (R10).
  \end{itemize}
\end{proof}

\appendix

\section{Remarks about the proof of Theorem~\ref{thm:Wald}}

\lbl{app:technical}

In this section we prove the technical results that are necessary for the proof
of Theorem~\ref{thm:Wald}.  We omit the main body of the proof, as the proof of
\cite[Theorem 4.2]{zakharevich10} works analogously here; however, we prove some
of the technical lemmas needed for that proof.

\begin{definition}
  The \textsl{set map} of a morphism $f:\SCob{A}{i} \rto \SCob{B}{j}$ in
  $\Tw(\C)$ is the underlying map of sets $f:I \rto J$.
\end{definition}

We have the following observation:

\begin{observation} \lbl{prop:fibers}
Given any diagram 
\begin{diagram}
{A & B \\ C & D \\};
\arrowsquare{f}{g}{h}{j}
\end{diagram}
where $D = \SCob{d}{l}$,
we can write it as a coproduct of diagrams of fibers
\begin{diagram}
{A_l & B_l \\ C_l & D_l \\};
\arrowsquare{f|_l}{g|_l}{h|_l}{j|_l}
\end{diagram}
Thus any pushout or pullback in $\Tw(\C)$ can be computed by computing it on
each fiber independently.
\end{observation}

The main technical work necessary for the proof of Theorem~\ref{thm:Wald} is the
construction of pushouts in $\SC(\C)$.  These depend on dependent products
in $\Tw(\C)$, which we construct in Lemma~\ref{lem:pullandpush}.  (For more
on dependent products, see \cite[Section IV]{maclanemoerdijk}.)  The functor
$\Pi_\sigma$, below, was denoted $\sigma_*$ in \cite{zakharevich10}.

We begin by showing that in many cases Grothendieck twists have dependent
products.

\begin{lemma} \lbl{lem:dependentproduct} Suppose that $\C$ is a category with
  all pullbacks and $\D$ is a sieve in $\C$.  Let $\sigma: A \rto B$ be a
  morphism in $\Tw(\iso\C\bs\D)$.  The functor $\sigma^*:(\Tw(\C\bs\D)/B) \rto
  (\Tw(\C\bs\D)/A)$ has a right adjoint $\Pi_\sigma$.  If $\sigma$ has an
  injective set map then $\sigma^*\Pi_\sigma \cong 1$; if $\sigma$ has a
  surjective set map then $\Pi_\sigma \sigma^* \cong 1$.
\end{lemma}

\begin{proof}
  Write $\sigma: A \rto B$, and let $\sigma^*: \Tw(\C\bs\D)/B \rto
  \Tw(\C\bs\D)/A$ be the functor pulling back along $\sigma$.  We want to show
  that $\sigma^*$ has a right adjoint $\Pi_\sigma$.  By
  Observation~\ref{prop:fibers} we can assume that $B = \{b\}$.
  
  Write $A = \SCob{a}{i}$.  If $I = \eset$ then there is exactly one object
  above $A$: $\{\}_\eset$.  Thus in this case we can define $\Pi_\sigma =
  B$ and the lemma clearly holds.  Thus from now on we can assume that $I\neq
  \eset$.  For any $p:A'\rto A$ with $A' = \SCob{A'}{k}$, write $A'_{(i)} =
  \{A'_k\}_{k\in p^{-1}(i)}$ for $i\in I$.  Let $\tilde A$ be the limit of the
  following diagram:
  \begin{squisheddiagram}[4em]
    {& \{a_i\} & A_{(i)}' \\
      \{b\} & \vdots & \vdots\\
      & \{a_{i'}\} & A'_{(i')} \\};
    \to{1-3}{1-2}_{p|_i} \to{3-3}{3-2}_{p|_{i'}}
    \to{1-2}{2-1}+{above}{\sigma_i} \to{3-2}{2-1}+{below}{\sigma_{i'}}
  \end{squisheddiagram}
  where the spokes are indexed over all elements of $I$.  (Note that the limit
  exists because it can be computed as an iterated pullback.)  We then define
  $\Pi_\sigma(p) = \tilde A \rto \{b\}$.
  
  We now need to show that $\Pi_\sigma$ is right adjoint to $\sigma^*$.  We 
  do this by showing that $\tilde A$ is the terminal object in $\sigma^*/A'$ for
  all $A'$.  Indeed, suppose that we have an object $(B'\rto^qB,p')$ in
  $\sigma^*/A'$.  Write $\sigma^*B' = \{B' \rto^q
  \{b\}\rto^{\{\Sigma_i^{-1}\}}\{a_i\}\}_{i\in I}$.  As $\sigma^*B'$ sits above
  $A'$ we are given morphisms $h_i$ such that for each $i$ the diagram
  \begin{diagram}
  {B' & \{b\} \\ A_{(i)}' & \{a_i\} \\};
  \arrowsquare{q}{h_i}{\{\Sigma_i^{-1}\}}{p|_i}
  \end{diagram}
  commutes.  However, this means that the diagram
  \begin{squisheddiagram}[4em]
    { & \{a_i\} & A'_{(i)} \\
      \{b\} & \vdots & \vdots & B' \\
      & \{a_{i'} & A'_{(i')} \\};
    \to{1-3}{1-2}_{p|_i} \to{3-3}{3-2}_{p|_{i'}}
    \to{1-2}{2-1}+{above}{\{\sigma_i\}\ } \to{3-2}{2-1}+{below}{\{\sigma_{i'}\}\
      \ 
    }
    \to{2-4}{1-3}+{above}{h_i} \to{2-4}{3-3}+{below}{h_{i'}}
  \end{squisheddiagram}
  commutes as well, which gives us a unique factorization of $\sigma^*B'$
  through $\tilde A$.  Thus $\tilde A$ is a terminal object and
  $\Pi_\sigma$ is right adjoint to $\sigma^*$.
    
  Now suppose that $\sigma$ has an injective set map.  This means that $I$ is
  either a singleton or empty, so we just need to prove that $\sigma^*\Pi_\sigma
  \cong 1$ in those cases.  If $I = \eset$ then $\sigma^*\Pi_\sigma \{\}_\eset =
  \{\}_\eset$, so $\sigma^*\Pi_\sigma \cong 1$ trivially.  If $I = \{*\}$ then
  \[\sigma^* \Pi_\sigma(B'\rto \{b\}) = \sigma^*(B' \rto \{b\}
  \rto^{\{\Sigma_*\}} \{a_*\}) = B' \rto \{b\},\] so $\sigma^*\Pi_\sigma \cong
  1$.
  
  On the other hand, suppose that $\sigma$ has a surjective set map.  Then we
  can assume that $I$ is nonempty, and write $\sigma^*(B'\rto\{b\}) =
  \coprod_{i\in I} (B'\rto \{b\}\rto^{\{\Sigma_i^{-1}\}} \{a_i\})$.  But then
  $B'$ is a pullback of the desired diagram, so we see that $\Pi_\sigma \sigma^*
  \cong 1$.
\end{proof}

\begin{lemma} \lbl{lem:subcomposition}
  Given two composable morphisms $f:A\rto B$ and $g:B\rto C$ in $\Tw(\C^\circ)$, if
  $gf$ is a sub-map then so is $f$.  If in addition $g$ and $gf$ are covering sub-maps, then so is $f$.
\end{lemma}

\begin{proof}
  Write $A = \SCob{a}{i}$, $B = \SCob{b}{j}$ and $C = \SCob{c}{k}$.  We need to
  show that if $f(i) = f(i')$ for $i,i'\in I$ then $a_i \times_{b_{f(i)}} a_{i'}
  = \initial$.  If $f(i) = f(i')$ then $gf(i) = gf(i')$, and as $gf$ is a
  sub-map $a_i \times_{c_{gf(i)}} a_{i'} = \initial$.  However, we have a
  morphism $a_i \times_{b_{f(i)}} a_{i'} \rto a_i \times_{c_{gf(i)}} a_{i'}$
  induced by the morphism $G_{f(i)}:b_{f(i)} \rto c_{gf(i)}$, so we conclude
  that $a_i\times_{b_{f(i)}} a_{i'} = \initial$.
  
  We now turn our attention to the second part of the lemma.  We first
  prove it in the case when $f$ is a section of $g$, so that $C=A$ and $gf =
  1_A$.  By Observation~\ref{prop:fibers} we can assume that $A = \{a\}$; we
  write the indexing set as $\{*\}$.  Then we have two composable morphisms
  in $\C$, 
  \[a \rto^{F_*} b_{f(*)} \rto^{G_{f(*)}} a,\]
  which compose to the
  identity.  By axiom (M), any such diagram must have both morphisms be
  isomorphisms.  As $g$ is a covering sub-map this means that $J = \{*\}$ as
  well, since otherwise $b_{f(*)} \times_a b_j \cong b_j \neq \initial$, which
  contradicts the sub-map condition.  Thus $f$ is an isomorphism, and in
  particular a covering sub-map.
  
  
  Now we consider the general case.  We have the following diagram,
  \begin{diagram}
	{A\times_C B & B \\ A & C\\};
	\cover{1-1}{1-2} \cover{2-1}{2-2}^{gf} \cover{1-2}{2-2}^g \cover{1-1}{2-1}_{t}
	\arrow{sub,->,out=135,in=-135}{2-1}{1-1}^s
	\sub{2-1}{1-2}!f
  \end{diagram}
  where $s$ is the section of $t:A\times_C B\rto A$ induced by $f$.  Applying
  the special case of the lemma to $s$ and $t$ we see that $s$ must be a
  covering sub-map, and thus $f$ is also one.
\end{proof}


The following is an immediate consequence of Lemma~\ref{lem:subcomposition}:

\begin{observation} \lbl{obs:preorder} For any object $A$ in $\Tw(\C^\circ)$,
  $\Tw(\C^\circ)^\SSub/A$ is a preorder of which $\W(\C)/A$ is a full
  subcategory.
\end{observation}

\extension{It feels like there should be a slightly more general version of
  this.  It can clearly be generalized to $\C\bs\D$, but is there more that can
  be done?}

Lemma~\ref{lem:dependentproduct} implies that Grothendieck twists of
assemblers have dependent products; it now remains only to check that they
preserve sub-maps and covering sub-maps.
  
\begin{lemma} \lbl{lem:pullandpush}
  Let $\C$ be a closed assembler.
Let $\sigma: A \rmove B$ be a move.  The functor 
\[\sigma^*:(\Tw(\C^\circ)/B) \rto (\Tw(\C^\circ)/A)\] has a right adjoint $\Pi_\sigma$, which also preserves (covering)
sub-maps.  These functors restrict to adjoint pairs
\[\sigma^*: \Tw(\C^\circ)^\SSub/B \rlto \Tw(\C^\circ)^\SSub/A \!\relcolon
\Pi_\sigma\]
and
\[\sigma^*: \W(\C)/B \rlto \W(\C)/A \!\relcolon
\Pi_\sigma.\]
\end{lemma}

\begin{proof}
  By Lemma~\ref{lem:dependentproduct} $\Pi_\sigma$ exists, so all we need to
  check is that it preserves (covering) sub-maps.  Using
  Lemma~\ref{lem:subcomposition}, in order to show that $\Pi_\sigma$ preserves
  (covering) sub-maps it suffices to show that if $A'\rsub^p A$ is a (covering)
  sub-map then so is $\Pi_\sigma(A' \rsub^p A)$.  By examining the proof of
  Lemma~\ref{lem:dependentproduct} we see that $\Pi_\sigma$ is computed as a
  pullback of morphisms
  \[A_i'\rto^{p|_i} \{a_i\} \rto^{\{\Sigma_i\}} \{b\}.\] If $p$ is a (covering)
  sub-map then so is each $p|_i$, and thus so is $\{\Sigma_i\}\circ p|_i$.
  Since $\Pi_\sigma$ is computed as the pullback of all of these and (covering)
  sub-maps are preserved under pullbacks, $\Pi_\sigma$ preserves (covering)
  sub-maps.

  We prove the restriction statement by showing that the unit and counit of
  the adjunction are both covering sub-maps.  Consider the unit; for a morphism
  $q:B' \rto B$ it is given by a diagram
  \begin{diagram}
    { B' && \Pi_\sigma \sigma^* B' \\ & B \\};
    \to{1-1}{1-3}^{\eta_q} \to{1-1}{2-2}_q \to{1-3}{2-2}^{\Pi_\sigma \sigma^*q}
  \end{diagram}
  If $q$ is a (covering) sub-map then so is $\Pi_\sigma \sigma^* q$, and thus by
  Lemma~\ref{lem:subcomposition} so is $\eta_q$.  The counit is handled
  analogously. 
\end{proof}

\bibliographystyle{IZ}
\bibliography{IZ-all}

\begin{thebibliography}{MLM94}

\bibitem[Bau91]{baues91}
Hans~Joachim Baues.
\newblock {\em Combinatorial homotopy and {$4$}-dimensional complexes},
  volume~2 of {\em de Gruyter Expositions in Mathematics}.
\newblock Walter de Gruyter \& Co., Berlin, 1991.
\newblock With a preface by Ronald Brown.

\bibitem[MLM94]{maclanemoerdijk}
Saunders Mac~Lane and Ieke Moerdijk.
\newblock {\em Sheaves in geometry and logic}.
\newblock Universitext. Springer-Verlag, New York, 1994.
\newblock A first introduction to topos theory, Corrected reprint of the 1992
  edition.

\bibitem[MT07]{murotonks07}
Fernando Muro and Andrew Tonks.
\newblock The 1-type of a {W}aldhausen {$K$}-theory spectrum.
\newblock {\em Adv. Math.}, 216(1):178--211, 2007.

\bibitem[MT08]{murotonks08}
Fernando Muro and Andrew Tonks.
\newblock On {$K_1$} of a {W}aldhausen category.
\newblock In {\em {$K$}-theory and noncommutative geometry}, EMS Ser. Congr.
  Rep., pages 91--115. Eur. Math. Soc., Z\"urich, 2008.

\bibitem[Vee84A]{veech84i}
William~A. Veech.
\newblock The metric theory of interval exchange transformations. {I}.
  {G}eneric spectral properties.
\newblock {\em Amer. J. Math.}, 106(6):1331--1359, 1984.

\bibitem[Vee84B]{veech84ii}
William~A. Veech.
\newblock The metric theory of interval exchange transformations. {II}.
  {A}pproximation by primitive interval exchanges.
\newblock {\em Amer. J. Math.}, 106(6):1361--1387, 1984.

\bibitem[Vee84C]{veech84iii}
William~A. Veech.
\newblock The metric theory of interval exchange transformations. {III}. {T}he
  {S}ah-{A}rnoux-{F}athi invariant.
\newblock {\em Amer. J. Math.}, 106(6):1389--1422, 1984.

\bibitem[ZakA]{Z-ass-var}
Inna Zakharevich.
\newblock The annihilator of the {L}efschetz motive.
\newblock http://arxiv.org/abs/1506.06200.

\bibitem[ZakB]{Z-enrich-wald}
Inna Zakharevich.
\newblock The category of {W}aldhausen categories as a closed multicategory.
\newblock http://arxiv.org/abs/1410.4834.

\bibitem[ZakC]{Z-Kth-ass}
Inna Zakharevich.
\newblock The {$K$}-theory of assemblers.
\newblock http://arxiv.org/abs/1401.3712.

\bibitem[Zak12]{zakharevich10}
Inna Zakharevich.
\newblock Scissors congruence as {$K$}-theory.
\newblock {\em Homology, Homotopy and Applications}, 14:181--202, 2012.

\bibitem[Zak13]{zakharevich11}
Inna Zakharevich.
\newblock Simplicial polytope complexes and deloopings of {$K$}-theory.
\newblock {\em Homology, Homotopy and Applications}, 15:301--330, 2013.

\end{thebibliography}

\end{document}